\documentclass[10pt,a4paper]{article}

\usepackage{graphicx}
\usepackage{color}
\usepackage{amssymb}
\usepackage{amsfonts}
\usepackage{amsthm}
\usepackage{amsmath}
\usepackage{latexsym}
\usepackage{mathrsfs}
\usepackage{mathtools}	
\usepackage{wasysym}
\usepackage{bbm}		
\usepackage{braket}		
\usepackage[inline]{enumitem}
\usepackage{caption}
\usepackage{subcaption}
\usepackage{nccmath}	
\usepackage{calc}
\usepackage{ifthen}
\usepackage[normalem]{ulem}
\usepackage[a4paper, left=2.7cm, right=2.7cm, top=3cm, bottom=2.3cm]{geometry}
\usepackage{graphicx}
\usepackage{soul}		
\usepackage{cancel}
\usepackage[pdftex=true,hyperindex=true,pdfborder={0 0 0}]{hyperref}
\usepackage{marginnote}
\usepackage{nicefrac}		
\usepackage{sectsty}		 
\usepackage[nottoc]{tocbibind}	 
\usepackage{doi}
\usepackage[numbers]{natbib}
\usepackage{tikz}
\usetikzlibrary{shapes,arrows,automata,matrix,fit,petri}

\numberwithin{equation}{section}

\newtheorem{theorem}{Theorem}[section]
\newtheorem{lemma}[theorem]{Lemma}

\newtheorem{proposition}[theorem]{Proposition}

\newtheorem{question}[theorem]{Question}
\newtheorem{assumption}[theorem]{Assumption}

\theoremstyle{definition}

\newtheorem{example}[theorem]{Example}

\newenvironment{case}[2][\unskip]{%
\noindent\textsl{Case #1}:\quad {#2}%
}{%
}

\newcommand{\exampleqed}{$\ocircle$\par}

\newcommand{\ZZ}{\mathbb{Z}}			
\newcommand{\NN}{\mathbb{N}}			
\newcommand{\RR}{\mathbb{R}}			
\newcommand{\isdef}{\coloneqq}			
\DeclarePairedDelimiter\abs{\lvert}{\rvert}		
\newcommand{\dd}{\mathrm{d}}			
\newcommand{\ee}{\mathrm{e}}			
\newcommand{\smallo}{o}					
\newcommand{\bigo}{O}					 
\renewcommand{\complement}{
\mathsf{c}%
}
\newcommand{\xPr}{\operatorname{\mathbb{P}}}		
\newcommand{\xExp}{\operatorname{\mathbb{E}}}		
\newcommand{\xVar}{\operatorname{\mathbb{V}\mathrm{ar}}}	
\newcommand{\indicator}[1]{\mathbbm{1}_{#1}}			
\newcommand{\rv}[1]{\pmb{#1}}			

\newcommand{\pspace}[1]{\mathscr{#1}}						
\newcommand{\effR}[3][\unskip]{\mathcal{R}^{#1}(#2\leftrightarrow#3)}	
\newcommand{\pathto}{\leadsto}							
\newcommand{\symb}[1]{\mathtt{#1}}						
\newcommand{\xadd}[1][\unskip]{
	\xrightarrow{\symb{+ #1\ifx#1\empty\else\;\fi}}%
}
\newcommand{\xremove}[1][\unskip]{
	\xrightarrow{\symb{- #1\ifx#1\empty\else\;\fi}}%
}
\newcommand{\critical}{\textrm{c}}							

\newcommand{\algor}{\mathfrak{A}}		

\newcommand*\circled[1]{
\tikz[baseline=-3pt]{
\node[shape=ellipse,draw,inner sep=0.5pt] (char) {\ensuremath{#1}};%
}%
}
\newcommand*\itemcircled[1]{%
\tikz[baseline=-3pt]{
\node[shape=ellipse,draw,inner sep=1pt] (char) {#1};%
}%
}
\newcommand*\itemboxed[1]{%
\tikz[baseline=-3pt]{
\node[shape=rectangle,draw,inner sep=2.5pt] (char) {#1};%
}%
}

\tikzstyle{site}=[inner sep=0pt,thick]
\tikzstyle{blue site}=[site,circle,draw=blue!75,fill=white,minimum size=10pt]
\tikzstyle{red site}=[site,diamond,draw=red!75,fill=white,minimum size=12pt]
\tikzstyle{particle}=[token=1]

\newcommand{\expgraph}{%
	\node[red site] (O) at (0,0) {};
	\node[blue site] (L) at (-3,0) {};
	\node[blue site] (R) at (3,0) {};
	\node[blue site] (U) at (0,3) {};
	\node[blue site] (D) at (0,-3) {};
	
	\node[blue site] (UL) at (-1,1) {};
	\node[red site] (ULL) at (-2,1) {};
	\node[red site] (UUL) at (-1,2) {};
	\node[blue site] (UULL) at (-2,2) {};
	
	\node[blue site] (UR) at (1,1) {};
	\node[red site] (URR) at (2,1) {};
	\node[red site] (UUR) at (1,2) {};
	\node[blue site] (UURR) at (2,2) {};
	
	\node[blue site] (DL) at (-1,-1) {};
	\node[red site] (DLL) at (-2,-1) {};
	\node[red site] (DDL) at (-1,-2) {};
	\node[blue site] (DDLL) at (-2,-2) {};
	
	\node[blue site] (DR) at (1,-1) {};
	\node[red site] (DRR) at (2,-1) {};
	\node[red site] (DDR) at (1,-2) {};
	\node[blue site] (DDRR) at (2,-2) {};
	
	\draw (O) edge (UL) edge (UR) edge (DL) edge (DR);
	\draw (U) edge (UUR) edge (UUL);
	\draw (D) edge (DDR) edge (DDL);
	\draw (R) edge (URR) edge (DRR);
	\draw (L) edge (ULL) edge (DLL);
	
	\draw (UULL) edge (UUL) edge (ULL);
	\draw (UL) edge (UUL) edge (ULL);
	
	\draw (UURR) edge (UUR) edge (URR);
	\draw (UR) edge (UUR) edge (URR);
	
	\draw (DDLL) edge (DDL) edge (DLL);
	\draw (DL) edge (DDL) edge (DLL);
	
	\draw (DDRR) edge (DDR) edge (DRR);
	\draw (DR) edge (DDR) edge (DRR);
}

\begin{document}

\title{Crossover times in bipartite networks\\ with activity constraints and time-varying switching rates}

\author{
Sem Borst\footnote{
Department of Mathematics and Computer Science,
Eindhoven University of Technology, Eindhoven, The Netherlands
}
\and
Frank den Hollander\footnote{
Mathematical Institute, Leiden University, Leiden, The Netherlands
}
\and
Francesca Nardi\footnote{%
	Francesca Nardi sadly passed away before the publication of this article.
}~\,\footnotemark[1]~\,\footnote{
Department of Mathematics, University of Florence, Florence, Italy
}
\and
Siamak Taati\footnotemark[2]~\,\footnote{
Bernoulli Institute, University of Groningen, Groningen, The Netherlands
}
}

\date{}

\maketitle

\begin{abstract}
In this paper we study the performance of a bipartite network in which customers arrive at the nodes of the network, but not all nodes are able to serve their customers at all times. Each node can be either active or inactive, and two nodes connected by a bond cannot be active simultaneously. This situation arises in wireless random-access networks where, due to destructive interference, stations that are close to each other cannot use the same frequency band. 

We consider a model where the network is bipartite, the active nodes switch themselves off at rate~$1$, and the inactive nodes switch themselves on at a rate that depends on time and on which half of the bipartite network they are in. An inactive node cannot become active when one of the nodes it is connected to by a bond is active. The switching protocol allows the nodes to share activity among each other. In the limit as the activation rate becomes large, we compute the crossover time between the two states where one half of the network is active and the other half is inactive. This allows us to assess the overall activity of the network depending on the switching protocol. Our results make use of the metastability analysis for hard-core interacting particle models on finite bipartite graphs derived in an earlier paper. They are valid for a large class of bipartite networks, subject to certain assumptions. Proofs rely on a comparison with switching protocols that are not time-varying, through coupling techniques. 

\medskip


\smallskip


\smallskip

\noindent
\emph{Acknowledgment:}
The research in this paper was supported through NWO Gravitation Grant 024.002.003--NETWORKS. ST was also supported through NWO grant 612.001.409.

\renewcommand{\contentsname}{\vspace{-1.5em}}
{\footnotesize\tableofcontents}
\end{abstract}


\section{Introduction}

Section~\ref{sec:intro:motivation} provides the motivation and background for our paper. Section~\ref{sec:intro:formulation} contains the mathematical formulation of the problem. Section~\ref{sec:intro:theorem} identifies the choices of the activation and deactivation rates in the switching protocol and formulates our main theorem for the crossover time.


\subsection{Motivation and background}
\label{sec:intro:motivation}

\paragraph*{Switching rates}
In the present paper we investigate metastability effects and hitting times for hard-core interaction dynamics with time-varying rates. Specifically, we consider a finite graph $G$ in which vertices (= nodes) can be either \emph{active} or \emph{inactive}, subject to the constraint that vertices connected by an edge (= bond) cannot be active simultaneously. Thus, the feasible joint activity states correspond to (the incidence vectors of) the independent sets of $G$, also called \emph{hard-core configurations}. We denote by $X(t)\in \{\symb{0},\symb{1}\}^{V(G)}$ (with $V(G)$ the vertex set of $G$) the joint activity state at time $t$, with $X_i(t)$ indicating whether vertex $i$ is inactive or active at time $t$. When vertex $i$ is inactive at time $t$, and none of its neighbours is active, it activates at a time-dependent exponential rate $\lambda_i(t)$. Activity durations are exponentially distributed with unit mean, i.e., when a vertex is active it deactivates at exponential rate $1$. Thus, $(X(t))_{t \geq 0}$ evolves as a time-inhomogeneous Markov process with state space $\pspace{X}\subseteq \{\symb{0},\symb{1}\}^{V(G)}$, with $\pspace{X}$ the set of hard-core configurations. 

We will examine the \emph{metastable} behaviour of $(X(t))_{t \geq 0}$ in an asymptotic regime where the activation rates $\lambda_i(t)$ grow large in a suitable sense. Metastable behaviour is different from \emph{mixing} behaviour, which concerns typical hitting times. Analysing metastability on random graphs is challenging, because complex geometric issues arise that require heavy mathematical machinery (see e.g.\ the monograph \cite{BovHol15}). Our specific interest is in \emph{time-dependent} transition rates, a setting that has not been considered in the literature.

\paragraph*{Random-access algorithms} 
The above-described problem is not only interesting from a methodological perspective, it is also relevant in analysing the performance of \emph{random-access algorithms in wireless networks}, in particular, so-called queue-based Carrier Sense Multiple Access (CSMA) policies. The activity periods in the hard-core interaction model correspond to the transmission times of data packets in the wireless network. The graph $G$ corresponds to the \emph{interference graph} of the wireless network, specifying which pairs of nodes are prevented from simultaneous transmission because of interference. In conventional CSMA policies, the various nodes activate at fixed rates, which gives rise to classical hard-core interactions models. Metastability characteristics and mixing properties of such models provide fundamental insight into starvation issues and performance characteristics in wireless networks. In particular, for high activation rates, the stationary distribution of the activity process concentrates on states where the maximum number of nodes is simultaneously active, with extremely slow transitions between them. This ensures high overall efficiency, but from the perspective of an individual node it induces prolonged periods of starvation, possibly interspersed with long sequences of transmissions in rapid succession, resulting in severe build-up of queues and long delays. We refer to~\cite{ZocBorLeeNar13}, \cite{NarZocBor15} and~\cite{Zoc15} for further background and a more comprehensive discussion of how the spatio-temporal dynamics of the activity process in wireless random-access networks can be represented in terms of hard-core interaction models.  We refer to \cite{JiaLecNiShiWal12} for a study of mixing times under CSMA scheduling. 

In queue-based CSMA policies, the activation rates are chosen to be functions of the queue lengths at the various nodes, with the aim to provide greater transmission opportunities to nodes with longer queues. Specifically, the activation rate typically increases as a function of the queue length at a node, and possibly decreases as a function of the queue lengths at neighbouring nodes. The activation rate would thus \emph{vary over time} as queues build up or drain when packets are generated or transmitted. For suitable activation rate functions, queue-based CSMA policies have been shown to achieve maximum stability, i.e., provide stable queues whenever feasible at all (see \cite{RajShaShi09,GhaSri10,JiaShaShiWal10,ShaShiTet11,ShaShi12} and reference therein). Hence, these policies have the capability to match the optimal throughput performance of centralised scheduling strategies, while requiring less computation and operating in a mostly distributed fashion. On the downside, the very activation rate functions required for ensuring maximum stability tend to result in long queues and poor delay performance (see \cite{BouBorLee14,GhaBorWhi14} and references therein). As alluded to above, metastability effects play a pivotal role in that regard, and analysing hitting times for the activity process $(X(t))_{t \geq 0}$ is critical in understanding, and possibly improving, the delay performance of queue-based CSMA policies.

\paragraph*{Bipartite inference graphs} 
In the present paper we focus on a finite \emph{bipartite} graph~$G$, whose vertex set can be partitioned into two sets $U$ and $V$ such that each edge connects one vertex in $U$ with one vertex in $V$ (and no two vertices within $U$ or within $V$).
Metastable behaviour of hard-core dynamics is most pronounced in bipartite graphs, is more amenable to a comprehensive analysis, and captures the essence of the metastability phenomenon.
Therefore, the case of bipartite graphs provides a natural initial stepping stone towards the analysis of more general graph structures.
As a crucial special case, the class of bipartite graphs includes grid graphs that have emerged as a canonical testing ground for exploring the delay performance of CSMA policies. Denote by $u \in \pspace{X}$ and $v \in \pspace{X}$ the joint activity states where all the vertices in either $U$ or $V$ are active, respectively. We will assume that the activation rates are of the form $\lambda_i(t) = \lambda_U(t)$ for all $i \in U$ and $\lambda_i(t) = \lambda_V(t)$ for all $i \in V$, where both $\lambda_U$ and $\lambda_V$ depend on a parameter $\lambda$ controlling the typical length of the queues (see~\eqref{eq:rates:choice} and~\eqref{eq:gdefs} for the choice of dependence to be considered). We are specifically interested in the asymptotic regime $\lambda \to \infty$ (which corresponds to a scenario with large queue lengths). We examine the distribution of the time $T_v = \inf\{t \geq 0\colon\,X(t) = v\}$ until state $v$ is reached for the first time when the system starts from state $u$ at time $0$.\footnote{The metastable behaviour and asymptotic distribution of $T_v$ when $\lambda_i(t) = \lambda^{1 + \alpha_U + \smallo(1)}$ for all $i \in U$ and $\lambda_i(t) = \lambda^{1 + \alpha_V + \smallo(1)}$ for all $i \in V$ were characterised by~den Hollander, Nardi and Taati~\cite{HolNarTaa16}.}

Even though in the above setting the activation rates do not explicitly depend on the queue lengths, the time-dependent rates $\lambda_U(t)$ and $\lambda_V(t)$ properly capture the relevant qualitative behaviour. Indeed, the joint activity states $u$ and $v$ will be asymptotically dominant as $\lambda \to \infty$, i.e., most of the time either all the nodes in $U$ or all the nodes in $V$ will be active. As a result, the queues of the nodes in $U$ and the queues of the nodes in $V$ will tend to either all increase or all decrease simultaneously. While the arrivals and transmissions of packets are governed by random processes, the trajectories of the queue lengths will be roughly linear when viewed on the long time scales of interest.\footnote{For the time-homogeneous setting this was proved in \cite{BorHolNarSfr18}.} Hence, under the assumption of identical arrival rates and initial queue lengths within the sets $U$ and $V$, queue-dependent activation rates can approximately be represented in terms of time-dependent activation rates, as specified above. Since the initial state is $X(0) = u$, the queues of the nodes in $U$ and the queues of the nodes in $V$ will initially tend to go down and up, respectively, and we therefore assume that $\lambda_U(\cdot)$ and $\lambda_V(\cdot)$ are decreasing and increasing functions, respectively (see Figure~\ref{fig:rates:choice:schematic}).


\begin{samepage}
\subsection{Mathematical formulation of the problem}
\label{sec:intro:formulation}
\paragraph*{The general model}
Let us now formulate the problem in more detail. As before, we consider a finite \emph{bipartite graph} $G$ as the underlying graph of the servers, consisting of two finite subsets of vertices $U$ and $V$. Whether a vertex $i\in U\cup V$ is inactive or active at time $t$ is specified by a Bernoulli random variable $X_i(t)\in\{\symb{0},\symb{1}\}$. For each vertex $i$ and each time $t$, we also have a random variable $Q_i(t)\in\NN_0 \isdef \{0,1,2,\ldots\}$ that denotes the length of the queue behind server $i$ at time $t$. The messages at server $i$ are served only during the periods in which $i$ is active. An active server turns inactive at rate $1$, while an inactive server $i$ attempts to become active at the ticks of an inhomogeneous Poisson process with rate $\lambda_i(t)$. An attempt at time $t$ is successful if none of the neighbours of $i$ are active at time $t^-$ (see Figure~\ref{fig:activation-and-conflicts}). All activation/inactivation attempts are independent. A \emph{random-access algorithm} uses the queue length of server $i$, and possibly the queue lengths at its set of neighbours $N(i)$, to decide the activation rate $\lambda_i(t)$, so that $\lambda_i(t) \isdef\algor[Q_i(t),Q_{N(i)}(t)]$ for some function $\algor$. Each such algorithm leads to a Markov process $(X(t),Q(t))_{t\geq 0}$ containing the activity state and the queue length of every server.
\end{samepage}

\begin{figure}[tbp]
\centering
{
\begin{tikzpicture}[xscale=1,yscale=1,every node/.style={scale=1},>=latex]
	\expgraph
	
	\node[particle] at (UUL) {};
	\node[particle] at (ULL) {};
	\node[particle] at (DLL) {};
	\node[particle] at (UR) {};
	\node[particle] at (DDRR) {};
	\node[particle] at (D) {};
	
\end{tikzpicture}
}
\caption{An example of a network with a bipartite underlying graph.  The dots indicate the active servers.  At the current state, the server in the middle cannot become active because its upper right neighbour is already active.}
\label{fig:activation-and-conflicts}
\end{figure}

\paragraph*{First and second approximations}
Our paper is motivated by random-access schemes that are assumed to be totally distributed, i.e., nodes do not have access to aggregate information about other nodes. In a first stage of approximation, we ignore the randomness of the queue lengths and assume that $Q_i(t)\in[0,\infty)$ increases with constant rate when $i$ is inactive and decreases with another constant rate when $i$ is active. In a second stage of approximation, we assume that $Q_i(t)$ is approximately the same for all vertices in the set $U$ or $V$ that $i$ lies in.
If we focus on the evolution of the Markov process starting from one of the two maximal packing configurations until the hitting time of the other maximal packing configuration, then we can assume that the functions $\lambda_i(t)$ are \emph{non-random} and are the same for all servers $i$ that are in $U$ or in $V$. In other words, our time-dependent activation rates may be interpreted as a proxy for queue-dependent activation rates.

\paragraph*{The time-inhomogeneous Markov process}
The approximated system can be thought of as a time-inhomogeneous Markov process $(X(t))_{t \geq 0}$ constructed as follows. (We use a similar representation as in the time-homogeneous setting.) The state space is the set $\pspace{X}$ of hard-core configurations on $G$. The process $(X(t))_{t \geq 0}$ is a c\`adla\`g process defined as follows. The transitions are triggered by a Poisson clock $\rv{\xi}$, i.e., a Poisson point process on $[0,\infty)$ with time-varying rate 
\begin{equation}
\label{eq:gammadef}
\gamma(t)=\big(1+\lambda_U(t)\big)\abs{U} + \big(1+\lambda_V(t)\big)\abs{V}.
\end{equation}
This clock is the union of the birth clocks (rate~$\lambda_U(t)$ or $\lambda_V(t)$ at each site in $U$ or $V$, respectively) and the death clocks (rate~$1$ at each site) for addition and removal of particles.  A transition at time $s\in\rv{\xi}$ is governed by a discrete transition kernel $K^{(s)}(\cdot,\cdot)$, which is essentially the transition matrix of the discrete hard-core dynamics studied in \cite{HolNarTaa16} with parameters $\lambda_U(s)$ and $\lambda_V(s)$. Namely, for distinct hard-core configurations $x$ and $y$, we have
\begin{align}
\label{eq:process:transition-matrix}
K^{(s)}(x,y) &\isdef
\begin{cases}
\frac{\lambda_U(s)}{\gamma(s)} 
&\text{if $x_i=\symb{0}$, $y_i=\symb{1}$, $x_{(U\cup V)\setminus\{i\}}=y_{(U\cup V)\setminus\{i\}}$
for some $i\in U$,} \\
\frac{\lambda_V(s)}{\gamma(s)} 
&\text{if $x_i=\symb{0}$, $y_i=\symb{1}$, $x_{(U\cup V)\setminus\{i\}}=y_{(U\cup V)\setminus\{i\}}$
for some $i\in V$,} \\
\frac{1}{\gamma(s)} 
&\text{if $x_i=\symb{1}$, $y_i=\symb{0}$, $x_{(U\cup V)\setminus\{i\}}=y_{(U\cup V)\setminus\{i\}}$
for some $i\in U\cup V$,} \\
0 
& \text{otherwise,}
\end{cases}
\end{align}
and $K^{(s)}(x,x)$ is defined so as to turn $K^{(s)}$ into a stochastic matrix. At every tick $s\in\rv{\xi}$ of the clock, the process jumps into a new state $X(s)$ distributed according to $K^{(s)}\left(X(s^-),\cdot\right)$.

Here is a more formal construction, which we need for coupling arguments. Let $\rv{\xi}$ be a Poisson process as above, and let $Z(x,t)$, $x\in\pspace{X}$, $t\in[0,\infty)$, be a collection of independent random variables in $\pspace{X}$, and independent of $\rv{\xi}$, with distribution $K^{(t)}(x,\cdot)$. Given an initial configuration $x(0)\in\pspace{X}$, the process $(X(t))_{t \geq 0}$ is constructed recursively by setting $X(s)\isdef Z\left(X(s^-),s\right)$ at each $s\in\rv{\xi}$.

As before, we write $u$ and $v$ for the configurations in which $U$ and $V$ are fully active, respectively. We are interested in the distribution of the hitting time $T_v\isdef\inf\{t\geq 0\colon\, X(t)=v\}$ conditional on starting at $X(0)=u$. Our goal will be to analyse the distribution of $T_v$.

\paragraph*{The regenerative structure of the process}
In the \emph{time-homogeneous} setting, we know that the hitting time $T_v$ starting from $u$ is approximately exponential in the asymptotic regime where $\lambda_U$ and $\lambda_V$ are large. The intuition comes from considering the return times to $u$ of the discrete-time embedded Markov chain as \emph{regeneration times}, and imagining a Bernoulli trial at each regeneration time. The trial is successful if the Markov chain visits $v$ before returning to $u$, and is unsuccessful otherwise. In the asymptotic regime, the probability of success in each trial is small and the expected duration of a single trial (successful or not) is negligible compared to the expected transition time $\xExp_u[T_v]$. The first success time of a large number of trials each having a small probability of success is approximately exponentially distributed (see e.g.,~\cite{Kei79,HolNarTaa16}).

In the \emph{time-inhomogeneous} setting we would like to follow a similar reasoning in order to show that, under appropriate conditions, the transition time $T_v$ from $u$ is approximately exponential with a non-constant rate. We can still use the returns to $u$ as regeneration times, but the success probability and the duration of the trials now depend on the starting times of the trials. One difficulty is that in \cite{HolNarTaa16} we only obtained information on the success probability and the duration of the trials in a time-homogeneous setting. We need to overcome this obstacle. 

To be more specific, let $\bar{\rv{\xi}}\isdef\{0\}\cup\rv{\xi}$. The times $s\in\bar{\rv{\xi}}$ at which $X(s)=u$ are considered as \emph{inhomogeneous regeneration times}. We denote the set of regeneration times by $\rv{\eta}$, so that $\rv{\eta}\isdef\{s\in\bar{\rv{\xi}}\colon\,X(s)=u\}$. A Bernoulli trial is made at each regeneration time $s\in\rv{\eta}$ with indicator random variable $B(s)$. For $t\in[0,\infty)$, we write
\begin{equation}
\begin{aligned}
T_v(t)       &\isdef \inf\{s\geq t\colon\,X(s)=v\}, \\
T^\circlearrowleft_u(t)	&\isdef \inf\{s>t\colon\, \text{$X(s)=u$ and $\rv{\xi}\left((t,s]\right)>0$}\},
\end{aligned}
\end{equation}
to denote the first hitting time of $v$ and the first return time of $u$ after time $t$. The Bernoulli random variable $B(s)$ is defined to be $1$ if $T_v(s)<T^\circlearrowleft_u(s)$ and to be $0$ otherwise. The success probability of the trial at $s\in\rv{\eta}$ is
\begin{equation}
\label{eq:varepsdef}
\varepsilon(s) \isdef \xPr(B(s)=1) = \xPr(T_v(s)<T^\circlearrowleft_u(s)\,|\,X(s)=u).
\end{equation}
The duration of the trial at $s\in\rv{\eta}$ is the random variable $\delta T(s)\isdef \min\{T_v(s),T^\circlearrowleft_u(s)\} - s$. This duration can also be measured in clock ticks by the discrete random variable $L(s)\isdef\rv{\xi}\big((s,s+\delta T(s)]\big)$ that counts the number of clock ticks from $s$ until the next regeneration time. Let
\begin{equation}
S \isdef \inf\{s\in\rv{\eta}: B(s)=1\}
\end{equation}
be the starting point of the first successful trial. The first hitting time of $v$ is the end point of the first successful trial, i.e., $T_v = S + \delta T(S)$. Our goal will be to analyse the distribution of $S$. Note that we expect $\delta T(S)$ to be negligible compared to $S$, as is the case in the time-homogeneous setting.

\paragraph*{Notation}
Throughout this paper we use the following notation:
\begin{itemize}
\item $f(x)\prec g(x)$ means $f(x)=\smallo(g(x))$ as $x\to\infty$,
\item $f(x)\preceq g(x)$ means $f(x)=\bigo(g(x))$ as $x\to\infty$,
\item $f(x)\asymp g(x)$ means $f(x)\preceq g(x)$ and $g(x)\preceq f(x)$ as $x\to\infty$.
\end{itemize}
Our analysis relies on various auxiliary random variables and couplings.
To avoid technical distractions, we describe these in an informal fashion
and use the same symbol $\xPr$ for the associated probability measures,
even when a formal definition would require distinct underlying probability spaces.
The notation $\xPr_u$ is used to indicate the law of the Markov chain
(homogeneous or inhomogeneous, depending on the context) starting from configuration~$u$.

\paragraph*{Review of some results from time-homogeneous setting}
The results of this paper are heavily based on comparison with the time-homogeneous version of the above Markov process in which the rates $\lambda_U$ and $\lambda_V$ are constant.  The time-homogeneous process was studied in~\cite{HolNarTaa16}, where detailed results were obtained on the mean and asymptotic distribution of the crossover time as well as the trajectory of the process close to its bottleneck (the formation of the ``critical droplet'').  The results of the latter paper were obtained for general finite bipartite graphs subject to certain hypotheses on the isoperimetric properties of the graph.  These hypotheses were verified for a few interesting classes of bipartite graphs, including the complete bipartite graph and the two-dimensional torus.  For the sake of comparison as well as future reference, we now briefly recall some relevant results from~\cite{HolNarTaa16}.

As before, let
\begin{align}
\varepsilon &\isdef \xPr_u(T_v<T^\circlearrowleft_u)
\end{align}
and let $\gamma$ denote the rate of the Poisson clock triggering the transitions given by the time-homogeneous version of~\eqref{eq:process:transition-matrix}.  Let
\begin{align}
\alpha &\isdef \frac{\log\lambda_V}{\log\lambda_U}-1 \;.
\end{align}
For a large family of bipartite graphs satisfying a mild \emph{isoperimetric property} (including the complete bipartite graph and the even torus), it was shown that, starting from $u$, the crossover time $T_v$ is asymptotically exponentially distributed, in the sense that
\begin{align}
\lim_{\lambda\to\infty}\xPr_u\big(T_v/\xExp_u[T_v] > t\big) &= \ee^{-t},
\end{align}
provided $\alpha>0$ and $\abs{U}<\big(1+\alpha\big)\abs{V}$.  Under the same assumptions,
\begin{align}
\label{eq:time-homogeneous:success-prob-vs-expectation}
\varepsilon\gamma\xExp_u[T_v] &= 1 + \smallo(1) \;,
\qquad \lambda\to\infty,
\end{align}
(see Equation~(A.5), Corollary~B.4 and Corollary~3.5 in~\cite{HolNarTaa16}). Furthermore,
\begin{align}
\label{eq:mean-crossover:asymptotics}
\xExp_u[T_v] &\asymp
\frac{\lambda_U^{\Delta(k^*)+k^*-1}}{\lambda_V^{k^*-1}} \;,
\qquad \lambda\to\infty,
\end{align}
(see Theorem~1.1 in~\cite{HolNarTaa16}), where $\Delta\colon\NN\to\NN$ denotes the (bipartite) \emph{isoperimetric cost} function of the graph (see Section~2.3 in~\cite{HolNarTaa16}) and $k^*$ is the \emph{critical size}, defined as the smallest positive integer maximising $\Delta(k)-\alpha(s)(k-1)$.

The case of a complete bipartite graph is relatively simple, and one can do direct calculations to obtain a sharp estimate
\begin{align}
\xExp_u[T_v] &= \frac{1}{\abs{U}}\lambda_U^{\abs{U}-1}[1 - \smallo(1)]\;,
\qquad \lambda\to\infty
\end{align}
(see Example~2.1 in~\cite{HolNarTaa16}).
For more general bipartite graphs, under more elaborate assumptions on the isoperimetric properties of the graph, one can obtain similar sharp estimates for the mean crossover time as well as detailed information about the trajectory of the process near its bottleneck, in particular, the shape of the ``critical droplet'' (Theorem~1.3 and Proposition~1.4 in~\cite{HolNarTaa16}).  As a prototypical example, in the case of a torus $\ZZ_m\times\ZZ_n$ (with $m,n$ even and nearest-neighbour edges), when $0<\alpha<1$, one finds that
\begin{align}
\label{eq:crossover:mean:torus}
\xExp_u[T_v] &= \frac{1}{4 m n \ell^*}\frac{
\lambda_U^{\ell^*(\ell^*+1)+1}
}{
\lambda_V^{\ell^*(\ell^*-1)}
}[1+\smallo(1)]\;, \qquad \lambda\to\infty,
\end{align}
where $\ell^*\isdef\lceil\nicefrac{1}{\alpha}\rceil$ is the size of the critical droplet.

\paragraph*{Stochastic ordering and monotonicity}
In~\cite{HolNarTaa16}, we strongly exploited an appropriate ordering on the configuration space $\pspace{X}$ and couplings respecting this ordering.  This ordering will also be crucial in the arguments of the current paper.

Let us define a partial order on $\pspace{X}$ by declaring $x\sqsubseteq x'$ if $x_U\supseteq x'_U$ and $x_V\subseteq x'_V$.  It is easy to verify that the Markov process (time-homogeneous or time-inhomogeneous) is \emph{monotonic} with respect to this partial order in the following sense: given $x,x'\in\pspace{X}$ with $x\sqsubseteq x'$, one can construct a coupling $\big(\big(X(t), X'(t)\big)\big)_{t\geq 0}$ of two copies of the process with $X(0)=x$ and $X'(0)=x'$ such that with probability~$1$, $X(t)\sqsubseteq X'(t)$ for every $t\geq 0$.

In fact, a more general statement is true.  Let $(\lambda_U,\lambda_V)$ and $(\lambda'_U,\lambda_V')$ be two choices of the parameters, each possibly varying with time, and assume that $\lambda_U\geq\lambda'_U$ and $\lambda_V\leq\lambda'_V$.  Then, given $x,x'\in\pspace{X}$ with $x\sqsubseteq x'$, one can construct a coupling $\big(\big(X(t), X'(t)\big)\big)_{t\geq 0}$ of the process with parameters $(\lambda_U,\lambda_V)$ and the process with parameters $(\lambda'_U,\lambda'_V)$ such that $X(0)=x$ and $X'(0)=x'$ and almost surely $X(t)\sqsubseteq X'(t)$ for every $t\geq 0$.  It follows that for every increasing event $E$ (i.e., an event such that $(y(t))_{t\geq 0}\in E$ whenever $(x(t))_{t\geq 0}\in E$ and $x(t)\sqsubseteq y(t)$ for all $t\geq 0$), it holds
\begin{align}
\xPr\big((X(t))_{t\geq 0}\in E\big) &\leq \xPr\big((X'(t))_{t\geq 0}\in E\big).
\end{align}
A frequent example of an increasing event in the present paper is the event that the process hits $v$ before returning to $u$.


\subsection{Choices of the activation rates and main theorems}
\label{sec:intro:theorem}

Given the above description, and in line with the follow-up paper~\cite{BorHolNarSfr18}, we assume that the activation rates are of the form
\begin{align}
\label{eq:rates:choice}
\lambda_U(t) &\isdef g_U([c_U\lambda - \mu_U t]^+) \;, &
\lambda_V(t) &\isdef g_V(c_V\lambda + \mu_V t) \;,
\end{align}
where $c_U,c_V,\mu_U,\mu_V$ are positive parameters, and both $g_U(x)$ and $g_V(x)$ are increasing with $g_V(x)\succ g_U(x)\succ 1$ as $x\to\infty$ (see Figure~\ref{fig:rates:choice:schematic}). The terms $c_U\lambda$ and $c_V\lambda$ represent approximate queue lengths of the servers in $U$ and $V$, respectively, at time $0$. The term $\mu_V$ represents the rate of arrival of new messages at servers in $V$ (which are inactive), while $\mu_U$ accounts for the service rate minus the rate of arrival of new messages at servers in $U$ (which are active). In~\cite{BorHolNarSfr18}, the model with these choices of parameters is referred to as the \emph{external model}, and a comparison is made between this model and the \emph{internal model}, which is the general model described at the beginning of Section~\ref{sec:intro:formulation}.

\begin{figure}[tbp]
\centering
{
\begin{tikzpicture}[scale=0.7,>=stealth']
\draw[->] (-0.2,0) -- (8.5,0) node[right] {$t$};
\draw[->] (0,-0.2) -- (0,4);
\draw[domain=0:8,color=red,very thick] plot (\x,{0.9*(1.2-0.08*\x)^2.7}) 
node[above left=-2pt,rotate=-3] {$\lambda_U(t)$};
\draw[domain=0:8,color=blue,very thick] plot (\x,{0.9*(1.2+0.04*\x)^3.3}) 
node[overlay,above left=-2pt,rotate=17] {$\lambda_V(t)$};
\end{tikzpicture}
}
\caption{A schematic graph of the activation rates as functions of time.}
\label{fig:rates:choice:schematic}
\end{figure}

\paragraph*{Target}
For concreteness, as in~\cite{BorHolNarSfr18}, we focus on the case in which
\begin{align}
\label{eq:gdefs}
g_U(x)&\asymp x^{\beta_U}\;, & g_V(x)&\asymp x^{\beta_V}\;,
\end{align}
where $\beta_V>\beta_U>0$. (Other choices for the functions $g_U,g_V$ would certainly be relevant for applications. Nonetheless, to limit the complexity of our analysis we restrict to polynomials.) 

We are concerned with the time scale at which the transition from $u$ to $v$ occurs when $\lambda$ is large, and with the limiting distribution of the transition time on this time scale.
Recalling~\eqref{eq:gammadef} and~\eqref{eq:varepsdef}, we let $\nu(s)\isdef\varepsilon(s)\gamma(s)$.
In light of the regeneration structure described above, it is natural to expect that,
for every time scale $M=M(\lambda)$ and every $\tau\in[0,\infty)$,
\begin{align}
\label{eq:target}
\lim_{\lambda\to\infty}\xPr_u\Big(\frac{T_v}{M}>\tau\Big) 
&= \begin{dcases}
	0		& \text{if $M\nu(M\tau)\succ 1$,} \\
	\exp\left(-\int_0^\tau \psi(\sigma)\,\dd\sigma\right) 	& \text{if $M\nu(M\tau)\asymp 1$,} \\
	1		& \text{if $M\nu(M\tau)\prec 1$,}
\end{dcases}
\end{align}
where in the middle case,
$\psi(\sigma)\isdef\lim_{\lambda\to\infty} M\nu(M\sigma)$.  We assume that the limit exists.
So, roughly speaking, if we let $M_\critical=M_\critical(\lambda)$ be the ``solution'' of $M\nu(M)\asymp 1$ (which is expected to be unique up to asymptotic equivalence $\asymp$), then the transition occurs almost surely on the time scale $M_\critical$, in the sense that  $\xPr_u(T_v>t) \approx 1$ for $t\prec M_\critical$ and $\xPr_u(T_v>t)\approx 0$ for $t\succ M_\critical$. On the time scale $M_\critical$, the transition time follows an exponential law with a time-varying rate. In~\cite{BorHolNarSfr18}, the equality $M\nu(M)\asymp 1$ is solved under the assumption that $G$ is a \emph{complete bipartite} graph and it is shown that the asymptotic behaviour of the system follows distinct patterns depending on whether $\beta_U<\frac{1}{\abs{U}-1}$, $\beta_U=\frac{1}{\abs{U}-1}$ or $\beta_U>\frac{1}{\abs{U}-1}$, which are referred to as the subcritical, the critical and the supercritical regime, respectively.  

The main goal of the present paper is to show that, under suitable conditions, a variant of the equality in~\eqref{eq:target} indeed holds. Throughout the sequel we make the following key assumptions.

\begin{assumption}[\textbf{Isoperimetric properties}]
\label{keyassump}
Either of the following two conditions holds:
\begin{enumerate}[label={\rm(\alph*)}] 
\item \label{keyassump:complete-bipartite} $G$ is a complete bipartite graph with $\abs{U}>1$.
\item \label{keyassump:general} $G$ satisfies hypothesis~\textup{(H2)} in~\textup{\cite{HolNarTaa16}} and $\beta_U\abs{U}<\beta_V\abs{V}$.
\end{enumerate}		
\end{assumption}

\begin{assumption}[\textbf{Energy barriers}]
\label{keyassumpextra}
$\check{\Gamma}^{(0)}\prec\sqrt{\Gamma^{(0)}}$. 
\end{assumption}

\noindent
Hypothesis (H2) in~\cite{HolNarTaa16} says that for every $i \in V$ there exists an isoperimetric numbering of~$V$ starting with $i$ that  is sufficiently long. The condition $\beta_U\abs{U}<\beta_V\abs{V}$ is a restatement of hypothesis~(H0) in~\cite{HolNarTaa16} and guarantees that $u$ is metastable and $v$ is stable, which is the setting considered in the present paper.
The quantities $\Gamma^{(s)}$ and $\check{\Gamma}^{(s)}$ are introduced in Section~\ref{sec:reghom} below.  Intuitively, for each $s \geq 0$, $\Gamma^{(s)}$ is the height of the hill that the time-homogeneous process with parameters $\lambda_U(s)$ and $\lambda_V(s)$ needs to climb in order to go from $u$ to $v$, whereas $\check{\Gamma}^{(s)}$ is the depth of the deepest well whose bottom is not $u$ or $v$ (the process may get stuck in this well on its way from $u$ to $v$).  Under Assumption~\ref{keyassump}, the identification of $\Gamma^{(s)}$ boils down to the identification of the (bipartite) isoperimetric cost function (see Section~\ref{sec:intro:formulation}). The quantity $\check{\Gamma}^{(s)}$ has not been studied before.  Assumption~\ref{keyassumpextra} is technical and presumably can be weakened. In Example~\ref{ex:cbg} below we show that for a complete bipartite graph it is redundant. 

In practice, it will be more convenient to replace $\varepsilon(s)$ in \eqref{eq:varepsdef} by
\begin{equation}
\check{\varepsilon}(s) \isdef \xPr_u^{(s)}(T_v<T^\circlearrowleft_u),
\end{equation}
where $\xPr^{(s)}$ is the law of a time-homogeneous Markov chain with parameters $\lambda_U(s)$ and $\lambda_V(s)$. The two quantities are expected to be close to each other, but the advantage of $\check{\varepsilon}(s)$ is that it is more tractable than $\varepsilon(s)$, and can be sharply estimated in various cases. Let 
\begin{equation}
\check{\nu}(s)\isdef\check{\varepsilon}(s)\gamma(s).
\end{equation}

\paragraph*{Main theorem}
The main theorem of the present paper is the following result identifying the asymptotic law of the crossover time.  

\begin{theorem}[\textbf{Law of crossover time}]
\label{thm:main}
Suppose that Assumptions~\textup{\ref{keyassump}--\ref{keyassumpextra}} are satisfied. Let $M=M(\lambda)>0$ be a given time scale such that $M(\lambda) \to \infty$ as $\lambda \to \infty$.
\begin{enumerate}[label=\textup{(\roman*)}]
\item \label{item:main:scenario-1} If $M\prec\lambda$ as $\lambda\to\infty$, then
\begin{align}
\label{eq:main-result}
\xPr_u\Big(\frac{T_v}{M}>\tau\Big) 
&= \begin{dcases}
	\smallo(1)		& \text{if $M\check{\nu}(0)\succ 1$,} \\
	[1\pm\smallo(1)]\exp\left(-\int_0^\tau M\check{\nu}(M\sigma)\,\dd\sigma\right)
		& \text{if $M\check{\nu}(0)\asymp 1$,} \\
	1-\smallo(1)	& \text{if $M\check{\nu}(0)\prec 1$,}
\end{dcases}
\end{align}
for every $\tau\in(0,\infty)$ as $\lambda\to\infty$.
\item \label{item:main:scenario-2} If $M\asymp\lambda$, then the identity~\eqref{eq:main-result} still holds for every $0<\tau<\frac{c_U}{\mu_U}\frac{\lambda}{M}$, while $\lim_{\lambda\to\infty}\xPr_u\big(T_v/M>\tau\big) = 0$ when $\tau\geq\frac{c_U}{\mu_U}\frac{\lambda}{M}$.
\item \label{item:main:scenario-3} If $M\succ\lambda$, then $\lim_{\lambda\to\infty} \xPr_u\big(T_v/M>\tau\big) = 0$ for every $\tau>0$.
\end{enumerate}
In fact, Assumption~\textup{\ref{keyassumpextra}} is only needed for the middle case of~\eqref{eq:main-result} in scenarios~\ref{item:main:scenario-1} and~\ref{item:main:scenario-2}.
\end{theorem}

The distinction between the \emph{three scenarios} $M\prec\lambda$, $M\asymp\lambda$ and $M\succ\lambda$ is simply due to the fact that, according to~\eqref{eq:rates:choice}, $\lambda_U(s)=0$ for $s\geq \frac{c_U}{\mu_U}\lambda$. Note that the conditions in~\eqref{eq:main-result} are in terms of $M\check{\nu}(0)$ rather than $M\check{\nu}(M\tau)$, as suggested by~\eqref{eq:target}.
As we will see in Proposition~\ref{prop:critical-regime:orders-of-magnitude}, the two quantities have the same order of magnitude when $M\tau<\frac{c_U}{\mu_U}\lambda$.
We have written~\eqref{eq:main-result} in the asymptotic form rather than the limit form appearing in~\eqref{eq:target} to allow for the possibility in the middle case that $M\check{\nu}(M\sigma)\asymp 1$ as $\lambda\to\infty$ even when $M\check{\nu}(M\sigma)$ does not converge to a function $\psi(\sigma)$.
We will refer to the top, middle and bottom cases in~\eqref{eq:main-result} as the \emph{supercritical}, \emph{critical} and \emph{subcritical} regime, respectively. It turns out that the subcritical and supercritical regime can be handled by direct comparison with the time-homogeneous setting.  We prove the identity for the critical regime by establishing tight lower and upper bounds for $\xPr_u(T_v>t)$.

The time-inhomogeneity of the system makes it difficult to verify the conditions of Theorem~\ref{thm:main} in full generality. Nevertheless, we show that the conditions are indeed satisfied in two special cases: when $G$ is a complete bipartite graph (i.e., the setting of~\cite{BorHolNarSfr18}) and when $G$ is an even torus $\ZZ_m\times\ZZ_n$ (two examples studied in~\cite{HolNarTaa16}).

\begin{example}[\textbf{Complete bipartite graph}]
\label{ex:cbg}
When $G$ is a complete bipartite graph, Assumption~\ref{keyassump} is trivially satisfied.  In Section~\ref{sec:reghom} (proof of Lemma~\ref{lem:no-deep-well}), we will verify that, for any $s \geq 0$,
\begin{align}
\Gamma^{(s)} \asymp \gamma(s)\lambda_U^{\abs{U}-1}(s), \quad
\check{\Gamma}^{(s)} \asymp \gamma(s)/\lambda_U(s), \qquad \lambda\to\infty.
\end{align}
For the choice of the functions $\lambda_U(t)$ and $\lambda_V(t)$ in \eqref{eq:rates:choice} and \eqref{eq:gdefs}, we have $\lambda_U(0) \asymp \lambda^{\beta_U}$, $\lambda_V(0) \asymp \lambda^{\beta_V}$ and $\gamma(0) \asymp \lambda^{\beta_U \vee \beta_V} = \lambda^{\beta_V} $ (recall that $\beta_V>\beta_U$), and so Assumption~\ref{keyassumpextra} is satisfied if and only if
\begin{align}
\label{eq:condCBG}
\beta_V < (\abs{U}+1)\beta_U.
\end{align}
However, we can remove this restriction by the following argument. When the system starts from $u$, $T_v$ is with high probability close to $T_\varnothing$, i.e., the first hitting time of the configuration~$\varnothing$ in which all the vertices are inactive. Note that in the trajectory from $u$ to~$\varnothing$ no vertex in $V$ ever gets an opportunity to become active. As a result, the asymptotic distribution of $T_v/M$ as $\lambda\to\infty$ (which is the same as the distribution $T_\varnothing/M$) is independent of $\beta_V$. Therefore, without loss of generality, we can lower $\beta_V$ so that it satisfies \eqref{eq:condCBG}. 

Let us now examine the quantity $\check{\nu}(s)=\check{\varepsilon}(s)\gamma(s)$.  As we will see in Section~\ref{sec:reghom}, $\check{\varepsilon}(s)\asymp 1/\Gamma^{(s)}$.  It follows that $\check{\nu}(s)=(c_U\lambda - \mu_U s)^{-(\abs{U}-1)\beta_U}$.  Therefore, in the scenario in which $M\prec\lambda$ and $\tau\in(0,\infty)$ or $M\asymp\lambda$ and $\tau\leq\frac{c_U}{\mu_U}\frac{\lambda}{M}$,
\begin{align}
\label{eq:main-result:complete-bipartite}
\lim_{\lambda\to\infty} \xPr_u\Big(\frac{T_v}{M}>\tau\Big)
&= {\begin{dcases}
0	& \text{if $M\succ \lambda^{(\abs{U}-1)\beta_U}$,} \\
\exp\left(-\int_0^\tau \psi(\sigma)\,\dd\sigma\right)	
& \text{if $M\asymp\lambda^{(\abs{U}-1)\beta_U}$,} \\
1	& \text{if $M\prec\lambda^{(\abs{U}-1)\beta_U}$.}
\end{dcases}}
\end{align}
where in the critical case,
$\psi(\sigma)\isdef\lim_{\lambda\to\infty}M(c_U\lambda - \mu_U M\sigma)^{-(\abs{U}-1)\beta_U}$.
We assume that the limit exists.

The choice of complete bipartite graph simplifies the geometry, because the metastable crossover is global rather than local. In \cite{BorHolNarSfr20}, the results for this example are used to analyse the case of \emph{general bipartite graphs}. It turns out that the metastable crossover consists of a succession of transitions on complete bipartite subgraphs. To identify which sequence of subgraphs the dynamics follows requires the use of a certain \emph{greedy algorithm} that captures the underlying \emph{geometric complexity}.
\hfill\exampleqed
\end{example}

\begin{example}[\textbf{Even torus}]
Let us now consider the case in which $G$ is the torus $\ZZ_m\times\ZZ_n$ with $m,n\in\NN$ even and nearest-neighbour edges.  Assume that $\beta_U\abs{U}<\beta_V\abs{V}$.  In~\cite[Section~6.1]{HolNarTaa16} it was verified that $G$ satisfies hypothesis~(H2), thus Assumption~\ref{keyassump} is satisfied. It follows from the results of that paper that
\begin{align}
\Gamma^{(0)} &\asymp \frac{
\lambda_U^{\ell^*(\ell^*+1)+1}(0)
}{
\lambda_V^{\ell^*(\ell^*-1)}(0)
} \asymp
\lambda^{\beta_U[\ell^*(\ell^*+1)+1] - \beta_V[\ell^*(\ell^*-1)]}.
\end{align}
This can also be seen by combining the fact that $\check{\varepsilon}(s)\asymp 1/\Gamma^{(s)}$ (see Section~\ref{sec:reghom}) with~\eqref{eq:time-homogeneous:success-prob-vs-expectation} and~\eqref{eq:crossover:mean:torus}.  We leave it as an open question to identify the order of magnitude of $\check{\Gamma}^{(0)}$ in this case.  Once such an estimate is available, one can identify the range of the parameters in which Assumption~\ref{keyassumpextra} is satisfied.

From~\eqref{eq:time-homogeneous:success-prob-vs-expectation} and~\eqref{eq:crossover:mean:torus}, we get
\begin{align}
\check{\nu}(s) &=
4 m n \ell^* \frac{
\lambda_V^{\ell^*(\ell^*-1)}(s)
}{
\lambda_U^{\ell^*(\ell^*+1)+1}(s)
}[1+\smallo(1)], \qquad \lambda\to\infty.
\end{align}
Thus, in the parameter regime in which Assumption~\ref{keyassumpextra} is satisfied, Theorem~\ref{thm:main} provides an explicit characterization of the asymptotic law of the crossover time for all choices of the time scale $M$.
\hfill\exampleqed
\end{example}

\paragraph*{Outline of remainder}
In Section~\ref{S2} we explain strategies to derive lower and upper bounds for the success time in a sequence of Bernoulli trials. In Section~\ref{S3} we use the latter to derive lower and upper bounds for the transition times in our network model in terms of certain key quantities. These quantities are further analysed in Section~\ref{S4}, and lead to explicit conditions on the model parameters under which Theorem~\ref{thm:main} can be proved.


\section{Exploiting the regenerative structure}
\label{S2}

In Section~\ref{S2.1} we reformulate the problem in terms of a sequence of Bernoulli trials and look at a simple case, formulated in Proposition~\ref{prop:renewal:simple} below. In Sections~\ref{S2.2} and \ref{S2.3} we derive lower and upper bound for the probability that the success time exceeds $t$, formulated in Propositions~\ref{prop:abstract:lower-bound} and~\ref{prop:abstract:upper-bound} below. In Section~\ref{S3} we will use the latter to formulate concrete bounds. 


\subsection{Reformulation}
\label{S2.1}

We begin by rephrasing the problem in abstract terms without referring to the underlying Markov process.

\paragraph*{General scenario}
We generate a sequence of Bernoulli trials, one after the other. Each trial has a random duration, so that the starting point of the $n$'th trial is random. The success probability and the length of each trial depend on its starting time, but are otherwise independent of the other trials. The outcome of a trial starting at time $s$ is indicated by a Bernoulli random variable $B(s)$, and its duration is denoted by $\delta T(s)$. So, if $0=S_0,S_1,S_2,\ldots$ are the starting times of the trials, then $S_{n+1}=S_n+\delta T(S_n)$. Let $S$ be the starting time of the first successful trial. What can we say about the distribution of $S$?

We are interested in an asymptotic regime where the success probabilities of the Bernoulli trials are small and the duration of each trial conditioned on its failure is approximately exponential with small mean. If $\varepsilon(s)\isdef\xPr(B(s)=1)$ is the success probability of the trial starting at time $s$ and $\gamma(s)$ is the approximate exponential rate of $\delta T(s)$ given $B(s)=0$, then we expect the success time $S$ to approximately have an inhomogeneous exponential distribution with rate $\varepsilon(s)\gamma(s)$, i.e.,
\begin{equation}
\label{eq:exponential:approximate}
\xPr(S>t) \approx \ee^{\textstyle{-\int_0^t \varepsilon(s)\gamma(s)\,\dd s}}.
\end{equation}
In the concrete setting explained above, we have a parameter $\lambda$, and as $\lambda\to\infty$ we expect $\varepsilon(s)=\smallo(1)$ to be close to the time-homogeneous setting with parameters $\lambda_U(s)$ and $\lambda_V(s)$. Conditional on $B(s)=0$, we also expect $\delta T(s)$ to be bounded from below by an exponential random variable with rate $\gamma(s)$, and to have expected value $[1+\smallo(1)]/\gamma(s)$.

\paragraph*{Scenario with underlying Poisson process}
We next concentrate on a more restricted setting that contains the concrete setting above. Let $\rv{\xi}$ be a Poisson point process on $[0,\infty)$ with time-varying rate function $\gamma(t)$ and set $\bar{\rv{\xi}}\isdef\{0\}\cup\rv{\xi}$. We assume that $t\mapsto\gamma(t)$ is integrable over any finite interval. For $s\in\bar{\rv{\xi}}$, let $B(s)$ be a Bernoulli random variable with parameter $\varepsilon(s)>0$, where $\varepsilon\colon\,[0,\infty)\to(0,1)$ is a sufficiently smooth and increasing function. For $s\in\rv{\xi}$, consider also a positive integer-valued random variable $L(s)$ that counts the duration of a potential trial at time $s$ in clock ticks. The random objects $\rv{\xi}$ and $B(s)$, $L(s)$ for $s\in\rv{\xi}$ do not need to be independent. However, we assume that conditional on $\rv{\xi}$ the pairs $(B(s),L(s))$ for different values of $s\in\rv{\xi}$ are independent.

The starting times of the trials can be identified recursively as follows. The first trial is made at time $S_0\isdef 0$. The $(n+1)$-st trial is made at time $S_{n+1}\isdef\sigma_{\rv{\xi}}(S_n,L(S_n))$, where $\sigma_{\rv{\xi}}(s,k)$ is the $k$'th tick of the clock $\rv{\xi}$ after time $s$. Let $\rv{\eta}\isdef\{S_0,S_1,S_2,\ldots\}\subseteq\bar{\rv{\xi}}$ be the random set of trial times. The first success time is $S\isdef\inf\{s\in\rv{\eta}\colon\, B(s)=1\}$.

\paragraph*{The simplest case}
The special case where $L(s)=1$ for every $s$ corresponds to having an exponential distribution with rate $\gamma(s)$ for the duration of each trial. If $B(s)$ is also independent of the Poisson process $\rv{\xi}$, then the distribution of the first hitting time is very close to an inhomogeneous exponential distribution.

\begin{proposition}[\textbf{Exponential duration}]
\label{prop:renewal:simple}
Suppose that $\xPr(L(s)=1\,|\,B(s)=0)=1$ for each $s\in\bar{\rv{\xi}}$, and that the Bernoulli random variables $B(s)$ are independent of the Poisson process $\rv{\xi}$. Then
\begin{equation}
\xPr(S>t) = \big(1-\varepsilon(0)\big)\,\ee^{\,\textstyle{-\int_0^t \varepsilon(s)\gamma(s)\,\dd s}}.
\end{equation}
\end{proposition}

\begin{proof}[Proof I (via coloring)]
This is immediate from the colouring theorem of the Poisson processes. Namely, let us colour each point $s\in\rv{\xi}$ blue if $B(s)=1$ and red otherwise. Since different points are coloured independently, the set of blue points in $\rv{\xi}$ is itself a Poisson process with rate function $\varepsilon(s)\gamma(s)$.
\end{proof}

\begin{proof}[Proof II (via Campbell's theorem)]
Conditioning on $\rv{\xi}$, we can write
\begin{equation}	
\begin{aligned}
\xPr(S>t\,|\,\rv{\xi}) &= \prod_{s\in\bar{\rv{\xi}}\cap[0,t]} \big(1-\varepsilon(s)\big) \\
&= \big(1-\varepsilon(0)\big)\,\ee^{\,\textstyle{\sum_{s\in\rv{\xi}\cap[0,t]}\log\big(1-\varepsilon(s)\big)}}
= \big(1-\varepsilon(0)\big)\,\ee^{\Sigma},
\end{aligned}
\end{equation}
where $\Sigma\isdef\sum_{s\in\rv{\xi}\cap[0,t]}f(s)$ with $f(s)\isdef\log(1-\varepsilon(s))$. According to Campbell's theorem (see Kingman~\cite[Section 3.2]{Kin93}), we have
\begin{align}
\xExp\big[\ee^\Sigma\big] &= \exp\left\{\int_0^t\big(\ee^{f(s)}-1\big)\gamma(s)\,\dd s\right\}
= -\int_0^t \varepsilon(s)\gamma(s)\,\dd s \in  (-\infty,0]. \nonumber
\end{align}
The latter integral is finite because $\gamma(s)$ is integrable on $[0,t]$ and $0\leq\varepsilon(s)\leq 1$. In summary, we obtain
\begin{equation}
\begin{aligned}
\xPr(S>t) &= \xExp\big[\xPr(S>t\,|\,\rv{\xi})\big] \\
&= \big(1-\varepsilon(0)\big)\,\xExp[\ee^\Sigma]
= \big(1-\varepsilon(0)\big)\,\ee^{\,\textstyle{-\int_0^t \varepsilon(s)\gamma(s)\,\dd s}},
\end{aligned}
\end{equation}
which proves the claim.
\end{proof}

\paragraph*{Making $B(s)$ and $L(s)$ independent}
Since all the trials whose durations are counted in $S$ are unsuccessful ($B(S)$ is the first successful trial), we can use the following trick to remove the dependence between $B(s)$ and $L(s)$. For $s\in\bar{\rv{\xi}}$, we construct a new random variable $\tilde{L}(s)$ as follows. For $B(s)=0$, set $\tilde{L}(s)\isdef L(s)$. Otherwise, choose $\tilde{L}(s)$ independently of $L(s)$ according to the distribution $\xPr(L(s)\in\cdot\,|\,B(s)=0)$. The new random variable $\tilde{L}(s)$ is independent of $B(s)$. Nevertheless, when using $\tilde{L}(s)$ instead of $L(s)$ we get the same value for the first success time $S$.


\subsection{A strategy to find lower bounds}
\label{S2.2}

Without the simplifying assumptions used above, we can still try to use Campbell's theorem to find lower bounds for $\xPr(S>t)$. To start, we can condition on both $\rv{\xi}$ and the collection $\tilde{L}(\cdot)$, and write
\begin{equation}
\begin{aligned}
\xPr\big(S>t\,\big|\, \rv{\xi},\tilde{L}(\cdot)\big) 
&= \prod_{s\in\rv{\eta}\cap[0,t]} \Big(1-\xPr\big(B(s)=1\,\big|\,\rv{\xi}\big)\Big) \\
&= \ee^{\,\textstyle{\sum_{s\in\rv{\eta}\cap[0,t]} \log \big(1-\xPr\big(B(s)=1\,\big|\,\rv{\xi}\big)\big)}} \\
&\geq \ee^{\,\textstyle{\sum_{s\in\bar{\rv{\xi}}\cap[0,t]} \log \big(1-\xPr\big(B(s)=1\,\big|\,\rv{\xi}\big)\big)}}.
\end{aligned}
\end{equation}
The last expression is independent of $\tilde{L}(\cdot)$, and so we get
\begin{equation}
\xPr(S>t\,|\, \rv{\xi}) \geq \ee^{\,\textstyle{\sum_{s\in\bar{\rv{\xi}}\cap[0,t]}
\log \big(1-\xPr\big(B(s)=1\,\big|\,\rv{\xi}\big)\big)}}.
\end{equation}
The latter expression looks like it can be integrated with the help of Campbell's theorem, except that $\xPr(B(s)=1\,|\,\rv{\xi})$ depends on $\rv{\xi}$. However, in our application the dependence is weak.

One idea is to use an upper bound for $\xPr\big(B(s)=1\,\big|\,\rv{\xi}\big)$ that is independent of $\rv{\xi}$. This gives a lower bound for $\xPr(S>t\,|\, \rv{\xi})$ in a form that is integrable by Campbell's theorem. However, in our application there does not seem to be any useful upper bound for $\xPr\big(B(s)=1\,\big|\,\rv{\xi}\big)$ that is valid almost surely. Namely, if $\rv{\xi}$ happens to have an unlikely large gap from $s$ onwards, then the probability $\xPr(B(s)=1\,\big|\,\rv{\xi})$ can be significant.

A more careful approach is to use an upper bound for $\xPr(B(s)=1\,\big|\,\rv{\xi})$ that holds most of the time. To be more specific, let $\hat{\varepsilon}(s)$ be a sufficiently smooth non-negative function, and let $N_t$ be the number of points $s\in\bar{\rv{\xi}}\cap[0,t]$ for which $\xPr(B(s)=1\,|\,\rv{\xi})>\hat{\varepsilon}(s)$. Assuming that $\delta_t(n)\isdef\xPr(N_t>n)$ decays rapidly, we hope to get a lower bound for $\xPr(S>t)$ of the form
\begin{equation}
\xPr(S>t) \geq \big(1-\alpha\big) \, \ee^{\,\textstyle{-\int_0^t \hat{\varepsilon}(s)\gamma(s)\,\dd s}}
\end{equation}
for some $0 < \alpha\ll 1$.

Suppose that when $\xPr(B(s)=1\,|\,\rv{\xi})>\hat{\varepsilon}(s)$ we have a (possibly worse) universal bound $\xPr(B(s)=1\,|\,\rv{\xi})<E$, where $E<1$ may depend on $t$ but not on $s$, at least when $\rv{\xi}$ is in a highly probable set $\Xi_t$. Then
\begin{equation}
\sum_{s\in\bar{\rv{\xi}}\cap[0,t]} \log\Big(1-\xPr\big(B(s)=1\,\big|\,\rv{\xi}\big)\Big)
\geq \sum_{s\in\bar{\rv{\xi}}\cap[0,t]}\log\big(1-\hat{\varepsilon}(s)\big) + N_t\,\log(1-E)
\end{equation}
when $\rv{\xi}\in\Xi_t$, which implies that
\begin{equation}
\xPr(S>t) \geq \xExp\Big[(1-E)^{N_t} \,\ee^{\,\sum_{s\in\bar{\rv{\xi}}\cap[0,t]}
\log(1-\hat{\varepsilon}(s))}\Big] - \xPr(\rv{\xi}\notin\Xi_t).
\end{equation}
The first term on the right-hand side has the form $\xExp\big[Z(1-E)^{N_t}\big]$ for a random variable $0<Z\leq 1$ and a non-negative integer-valued random variable $N_t$ with a rapidly decaying tail $\delta_t(n)=\xPr(N_t>n)$. Campbell's theorem can be used to integrate $Z$ alone, but it is not clear how we can integrate the product of $Z$ and $(1-E)^{N_t}$.

We split $\xExp\big[Z(1-E)^{N_t}\big]$ based on whether $N_t$ is larger or smaller than a constant $m\geq 0$, which we will need to choose later:
\begin{equation}
\label{eq:lbound:split}
\begin{aligned}
\xExp\big[Z(1-E)^{N_t}\big] &\geq
\xExp\big[Z(1-E)^m\,\indicator{N_t\leq m}\big] +
\xExp\big[Z(1-E)^{N_t}\,\indicator{N_t>m}\big] \\
&= \xExp\big[Z(1-E)^m\big] 
- \xExp\Big[Z\underbrace{\big((1-E)^m-(1-E)^{N_t}\big)}_{\leq 1}\,\indicator{N_t>m}\Big] \\
&\geq (1-E)^m\xExp[Z] - \xExp\big[Z\,\indicator{N_t>m}\big].
\end{aligned}
\end{equation}
Applying the Cauchy-Schwarz inequality to the second term, we have
\begin{equation}
\label{eq:lbound:cauchy-schwarz}
\xExp\big[Z\,\indicator{N_t>m}\big] 
\leq \xExp[Z^2]^{\nicefrac{1}{2}} \xExp[\indicator{N_t>m}]^{\nicefrac{1}{2}}
=\sqrt{\delta_t(m)}\,\xExp[Z^2]^{\nicefrac{1}{2}}.
\end{equation}
Now, Campbell's theorem allows us to calculate 
\begin{equation}
\label{eq:lbound:Z}
\xExp[Z] = \big(1-\hat{\varepsilon}(0)\big)\,
\ee^{\,\textstyle{-\int_0^t \hat{\varepsilon}(s)\gamma(s)\,\dd s}}
\end{equation}
and 
\begin{equation}
\label{eq:lbound:Z2}
\begin{aligned}
\xExp[Z^2] &= \big(1-\hat{\varepsilon}(0)\big)^2\,
\ee^{\,\textstyle{-\int_0^t \big([1-\hat{\varepsilon}(s)]^2-1\big)\gamma(s)\,\dd s}}\\
&= \big(1-\hat{\varepsilon}(0)\big)^2\,\ee^{\,\textstyle{-2\int_0^t
\hat{\varepsilon}(s)\big[1-\frac{1}{2}\hat{\varepsilon}(s)\big] \gamma(s)\,\dd s}}.
\end{aligned}
\end{equation}
Combining \eqref{eq:lbound:split}--\eqref{eq:lbound:Z2}, we get
\begin{equation}
\begin{aligned}
\xExp\big[Z(1-E)^{N_t}\big] &\geq
\begin{multlined}[t]
(1-E)^m\big(1-\hat{\varepsilon}(0)\big)\,
\ee^{\,\textstyle{-\int_0^t \hat{\varepsilon}(s)\gamma(s)\,\dd s}} \\
- \sqrt{\delta_t(m)}\,\big(1-\hat{\varepsilon}(0)\big)\,\ee^{\,\textstyle{-\int_0^t
\hat{\varepsilon}(s)\big[1-\frac{1}{2}\hat{\varepsilon}(s)\big]\gamma(s)\,\dd s}}
\end{multlined} \\
&= \underbrace{\big(1-\hat{\varepsilon}(0)\big)}_{\approx 1}\,
\Big[\underbrace{(1-E)^m}_{\approx 1} - \underbrace{\sqrt{\delta_t(m)}}_{\ll 1}
\underbrace{\ee^{\,\textstyle{\frac{1}{2}\int_0^t 
\hat{\varepsilon}(s)^2\gamma(s)\,\dd s}}}_{\approx 1}\Big]\,
\ee^{\,\textstyle{-\int_0^t \hat{\varepsilon}(s)\gamma(s)\,\dd s}},
\end{aligned}
\end{equation}
which has the desired form. In summary, we have the following proposition.

\begin{proposition}[\textbf{Lower bound}]
\label{prop:abstract:lower-bound}
Let $\hat{\varepsilon}(s)$ be a positive measurable function, and let $N_t$ be the number of points $s\in\bar{\rv{\xi}}\cap[0,t]$ for which the inequality $\xPr\big(B(s)=1\,\big|\,\rv{\xi}\big)\leq\hat{\varepsilon}(s)$ fails. Let $0<E<1$ be a constant such that $\xPr\big(B(s)=1\,\big|\,\rv{\xi}\big)\leq E$ whenever $\rv{\xi}$ is in a measurable set $\Xi_t$ and $s\in \bar{\rv{\xi}}\cap[0,t]$. Then, for every $m\geq 0$,
\begin{align}
\xPr(S>t) &\geq
\check{K}(m)\, \ee^{\,\textstyle{-\int_0^t \hat{\varepsilon}(s)\gamma(s)\,\dd s}}
- \xPr(\rv{\xi}\notin\Xi_t) \;,
\end{align}
where
\begin{align}
\check{K}(m) &\isdef \big(1-\hat{\varepsilon}(0)\big) \,
\Big[(1-E)^m - \sqrt{\delta_t(m)}\,\ee^{\,\textstyle{\frac{1}{2}\int_0^t
\hat{\varepsilon}(s)^2\gamma(s)\,\dd s}}\Big]
\end{align}
and $\delta_t(m)\isdef\xPr(N_t>m)$.
\end{proposition}

Think of $\hat{\varepsilon}(s)$ as a good typical bound and of $E$ as a rough universal bound.

\subsection{A strategy to find upper bounds}
\label{S2.3}

We can try a similar approach via Campbell's theorem to find an upper bound for $\xPr(S>t)$. Conditioning on $\rv{\xi}$ and $\tilde{L}(\cdot)$ as before, we write
\begin{equation}
\begin{aligned}
\xPr\big(S>t\,\big|\, \rv{\xi},\tilde{L}(\cdot)\big) 
&= \prod_{s\in\rv{\eta}\cap[0,t]} \Big(1-\xPr\big(B(s)=1\,\big|\,\rv{\xi}\big)\Big) \\
&= \ee^{\,\textstyle{\sum_{s\in\rv{\eta}\cap[0,t]} \log \big(1-\xPr\big(B(s)=1\,\big|\,\rv{\xi}\big)\big)}}.
\end{aligned}
\end{equation}
We have $\xPr(B(s)=1\,\big|\,\rv{\xi})\geq \check{\varepsilon}(s)$, which holds independently of $\rv{\xi}$. Namely, $\check{\varepsilon}(s)$, which is a positive measurable function, is the probability of $B(s)=1$ when we freeze the parameters $\lambda_U$ and $\lambda_V$ at time $s$. Using this bound, we have
\begin{equation}
\xPr\big(S>t\,\big|\, \rv{\xi},\tilde{L}(\cdot)\big) \leq \ee^{\,\textstyle{\sum_{s\in\rv{\eta}\cap[0,t]}
\log \big(1-\check{\varepsilon}(s)\big)}}.
\end{equation}
We can bound the sum on the right-hand side, noting that all the terms are negative, as
\begin{equation}
\begin{aligned}
\Sigma\big(\rv{\xi},\tilde{L}(\cdot)\big) &\isdef \sum_{s\in\rv{\eta}\cap[0,t]}
\log\big(1-\check{\varepsilon}(s)\big) \\
&\leq \sum_{s\in\bar{\rv{\xi}}\cap[0,t]} \log\big(1-\check{\varepsilon}(s)\big)
- \sum_{s\in\bar{\rv{\xi}}\cap[0,t]}
\sum_{\substack{r\in\bar{\rv{\xi}}\cap[0,t]\\ s<r<\sigma_{\rv{\xi}}(s,\tilde{L}(s))}}
\log\big(1-\check{\varepsilon}(r)\big),
\end{aligned}
\end{equation}
where, as before, $\sigma_{\rv{\xi}}(s,k)$ denotes the $k$'th element of the clock $\rv{\xi}$ after time $s$. Suppose that $0<E'<1$ (possibly dependent on $t$) is such that $\check{\varepsilon}(s)\leq E'$ for each $s\in[0,t]$. In our concrete setting, $\check{\varepsilon}(s)$ is non-decreasing and we can choose $E'\isdef\check{\varepsilon}(t)$. Replacing $\check{\varepsilon}(r)$ by $E'$ in the above inequality, we get
\begin{equation}
\Sigma\big(\rv{\xi},\tilde{L}(\cdot)\big) \leq
\sum_{s\in\bar{\rv{\xi}}\cap[0,t]} \log\big(1-\check{\varepsilon}(s)\big)
- \sum_{s\in\bar{\rv{\xi}}\cap[0,t]} (\tilde{L}(s)-1)\log\big(1-E'\big).
\end{equation}
Integrating with respect to $\tilde{L}(\cdot)$, we get
\begin{equation}
\begin{aligned}
\xPr(S>t\,|\, \rv{\xi}) 
&= \xExp\Big[\xPr\big(S>t\,\big|\, \rv{\xi},\tilde{L}(\cdot)\big) \,\Big|\,\rv{\xi}\Big]
\leq \xExp\big[\ee^{\Sigma\big(\rv{\xi},\tilde{L}(\cdot)\big)}\,\big|\,\rv{\xi}\big] \\
&\leq \ee^{\,\textstyle{\sum_{s\in\bar{\rv{\xi}}\cap[0,t]}
\log \big(1-\check{\varepsilon}(s)\big)}}
\xExp\left[\quad\prod_{\mathclap{s\in\bar{\rv{\xi}}\cap[0,t]}}\;
(1-E')^{-(\tilde{L}(s)-1)}\,\middle|\,\rv{\xi}\right].
\end{aligned}
\end{equation}
Recall that, conditional on $\rv{\xi}$, the random variables $\tilde{L}(s)$ for different values of $s\in\bar{\rv{\xi}}$ are independent. Therefore we can take the product out of the expectation and write
\begin{equation}
\xPr(S>t\,|\, \rv{\xi}) \leq \ee^{\,\textstyle{\sum_{s\in\bar{\rv{\xi}}\cap[0,t]}
\log \big(1-\check{\varepsilon}(s)\big)}}
\;\prod_{\mathclap{s\in\bar{\rv{\xi}}\cap[0,t]}}\;
\xExp\big[(1-E')^{-(\tilde{L}(s)-1)}\,\big|\,\rv{\xi}\big].
\end{equation}
Suppose that we can find a good bound $\xExp\big[(1-E')^{-(\tilde{L}(s)-1)}\,\big|\,\rv{\xi}\big]\leq\rho(s)$ that holds whenever $\rv{\xi}$ is in a highly probably set $\Xi_t$, where $\rho(s)$ is a measurable function not depending on $\rv{\xi}$. (In particular, we would like to have $\rho(s)-1\ll\check{\varepsilon}(s)$ or at least $\rho(s)-1\ll E'$. The existence of such a bound is plausible, because $E'\ll 1$ and $\tilde{L}(s)$ is expected to be close to $1$ with high probability and in expectation.) Then we obtain the bound
\begin{align}
\xPr(S>t\,|\, \rv{\xi}) &\leq
\ee^{\,\textstyle{\sum_{s\in\bar{\rv{\xi}}\cap[0,t]}
\big[\log \big(1-\check{\varepsilon}(s)\big)+\log\rho(s)\big]}}
+ \indicator{\Xi_t^\complement}(\rv{\xi}) \;,
\end{align}
which is integrable via Campbell's theorem. Namely, we get
\begin{equation}
\begin{aligned}
\xPr(S>t) 
&\leq \rho(0)\,\big(1-\check{\varepsilon}(0)\big)\,\ee^{\,\textstyle{\int_0^t
\big(\rho(s)\big[1-\check{\varepsilon}(s)\big]-1\big)\gamma(s)\dd s}} + \xPr(\rv{\xi}\notin\Xi_t) \\
&\leq \Big[\rho(0)\,\big(1-\check{\varepsilon}(0)\big)\,\ee^{\,\textstyle{\int_0^t
\big(\rho(s)-1\big)\gamma(s)\dd s}}\Big] \,
\ee^{\,\textstyle{-\int_0^t \check{\varepsilon}(s)\gamma(s)\dd s}} + \xPr(\rv{\xi}\notin\Xi_t).
\end{aligned}
\end{equation}
Note that if $\rho(s)-1\ll\check{\varepsilon}(s)$, then the factor in the bracket is close to $1$. In summary, we have the following proposition.

\begin{proposition}[\textbf{Upper bound}]
\label{prop:abstract:upper-bound}
Let $\check{\varepsilon}(s)$ be a positive measurable function such that $\xPr\big(B(s)=1\,\big|\,\rv{\xi}=\xi\big)\geq\check{\varepsilon}(s)$ almost surely for all $s\in\bar{\rv{\xi}}\cap[0,t]$. Let $0<E'<1$ be a constant such that $\check{\varepsilon}(s)\leq E'$ for each $s\in[0,t]$. Moreover, let $\rho(s)\geq 1$ be a measurable function such that $\xExp\big[(1-E')^{-(\tilde{L}(s)-1)}\,\big|\,\rv{\xi}=\xi\big]\leq\rho(s)$ for every $\xi$ in a measurable set $\Xi_t$ and all $s\in\bar{\xi}\cap[0,t]$. Then
\begin{align}
\xPr(S>t) &\leq
\hat{K} \,\ee^{\,\textstyle{-\int_0^t \check{\varepsilon}(s)\gamma(s)\dd s}}
+ \xPr(\rv{\xi}\notin\Xi_t),
\end{align}
where
\begin{align}
\hat{K} &\isdef \rho(0)\,\big(1-\check{\varepsilon}(0)\big)\,
\ee^{\,\textstyle{\int_0^t\big(\rho(s)-1\big)\gamma(s)\dd s}}.
\end{align}
\end{proposition}


\section{Back to hard-core dynamics}
\label{S3}

Throughout this section, we consider a time scaling of the form $t=M(\lambda)\tau$, where $M=M(\lambda)$ is a positive function that tends to $\infty$ as $\lambda\to\infty$ and $\tau\geq 0$ is the scaled time. In the concrete setting of the time-inhomogeneous hard-core dynamics, we wish to use Propositions~\ref{prop:abstract:lower-bound} and~\ref{prop:abstract:upper-bound} to find sharp bounds for the probability $\xPr(S>t)$. Two questions arise:

\begin{question}
How should we choose $\hat{\varepsilon}(s)$, $E$ and $m$? In particular, we want $\hat{\varepsilon}(s)=\varepsilon(s)[1+\smallo(1)]$ as $\lambda\to\infty$, $E=\smallo(1)$ as $\lambda\to\infty$, and $0<\delta_t(m)\ll 1$, preferably $\delta_t(m)=\smallo(1)$ as $\lambda\to\infty$.
\end{question}

\begin{question}
Can we find a good upper bound for $\xExp\big[r^{\tilde{L}(s)-1}\,\big|\,\rv{\xi}\big]\leq\rho(s)$ for $r\isdef \frac{1}{1-E'}>1$?
\end{question}

In Sections~\ref{S3.1} and \ref{S3.2} we answer these questions, in the form of Propositions~\ref{prop:upper-bound:simplified} and \ref{prop:lower-bound:typical} below. In Section~\ref{S3.3} this leads to a set of further tasks, summarized in Proposition~\ref{prop:remaining-tasks} below, which we address in Section~\ref{S4}.

\paragraph*{Freezing the parameters at time $t=M\tau$}
Observe that the event $T_v>M\tau$ depends only on the state of the process up until time~$M\tau$.  Therefore, the probability $\xPr_u(T_v>M\tau)$ is independent of the value of the parameters $\lambda_U(s)$ and $\lambda_V(s)$ for $s>M\tau$.  In the remainder of the paper, we consider a modified version of the process in which $\lambda_U(s)$ and $\lambda_V(s)$ are truncated at time $s=M\tau$, thus assuming
\begin{align}
\label{eq:assumption:freezing-at-the-end}
	\left\{
		\begin{array}{@{}l@{}l}
			\lambda_U(s) &=\lambda_U(M\tau) \\ [1ex]
			\lambda_V(s) &=\lambda_V(M\tau)
		\end{array}
	\right.
	\qquad
	\forall\,s\geq M.
\end{align}
This assumption does not affect the validity of Theorem~\ref{thm:main} but will simplify the presentation of its proof.


\subsection{Simplified condition for the lower bound}
\label{S3.1}

In order to apply Proposition~\ref{prop:abstract:lower-bound}, we need a good typical upper bound $\hat{\varepsilon}(s)$ for $\xPr\big(B(s)=1\,\big|\,\rv{\xi}=\xi\big)$, where ``typical'' means for most $\xi$.  We also need a rough upper bound $E$ for $\xPr\big(B(s)=1\,\big|\,\rv{\xi}=\xi\big)$ that holds uniformly in $\xi$, in a highly probable set $\Xi_t$, and uniformly in $s\in\bar{\xi}\cap[0,t]$.

To obtain a typical upper bound $\hat{\varepsilon}(s)$, we choose a value $\delta s>0$ (depending on $\lambda$) that is small compared to the overal duration $t$, but large enough such that typically the clock $\rv{\xi}$ has at least $k$ ticks within the interval $(s,s+\delta s)$ for some $k\in\ZZ^+$. The parameter $k$ is chosen large enough so that the Markov chain typically takes less than $k$ steps to reach either $u$ or $v$, i.e., a trial starting at $s$ would typically end before time $s+\delta s$, and so we can bound its probability by coupling the process with the one having frozen parameters $\lambda_U(s+\delta s)$ and $\lambda_V(s+\delta s)$.

To make this idea precise, let $\delta\hat{T}_v(s)$ denote the number of ticks of $\xi$ from time $s$ until the first time the discrete Markov chain arrives at $v$, and define $\delta\hat{T}^\circlearrowleft_u(s)$ similarly.

\begin{proposition}[\textbf{Simplified lower bound}]
\label{prop:lower-bound:typical}
Let $t=M(\lambda)\tau$, where $M=M(\lambda)\to\infty$ as $\lambda\to\infty$ and $\tau\geq 0$ is a constant. Let $\delta s>0$ and $k\in\ZZ^+$ (each possibly depending on $\lambda$). Suppose that, uniformly in $\xi$ and $s\in\bar{\xi}\cap[0,t]$,
\begin{align}
\xPr\big(\delta\hat{T}_v(s)>k\,\big|\ \rv{\xi}=\xi,\,X(s)=u,\,\;\delta\hat{T}_v(s)<\delta\hat{T}^\circlearrowleft_u(s)\big)
&= \smallo(1), \qquad\lambda\to\infty.
\end{align}
Then
\begin{align}
\xPr\big(B(s)=1\,\big|\,\rv{\xi}=\xi\big) 
&= \xPr_u^{(s+\delta s)}(T_v<T^\circlearrowleft_u)\,
[1+\smallo(1)], \qquad \lambda\to\infty. 
\end{align}
uniformly in $\xi$ and $s\in\bar{\xi}\cap[0,t]$ satisfying $\xi(s,s+\delta s)\geq k$.
\end{proposition}

\begin{proof}
Let $\xi$ be fixed and consider a point $s\in\bar{\xi}\cap[0,t]$. The success probability of the trial starting at $s$ can be bounded as follows. Note that
\begin{align}
\xPr\big(B(s)=1\,\big|\,\rv{\xi}=\xi\big) 
&= \xPr\big(T_v(s)<T^\circlearrowleft_u(s)\,\big|\,\rv{\xi}=\xi,X(s)=u\big) \\
&= \xPr\big(T_v(s)<T^\circlearrowleft_u(s),\; T_v(s)\leq s+\delta s\,\big|\,\rv{\xi}=\xi,X(s)=u\big) \nonumber\\
& \quad + \xPr\big(s+\delta s<T_v(s)<T^\circlearrowleft_u(s)\,\big|\,\rv{\xi}=\xi,X(s)=u\big) \nonumber. 
\end{align}
For the first term, using the coupling argument discussed at the end of Section~\ref{sec:intro:formulation}, we can show that
\begin{align}
\MoveEqLeft
\xPr\big(T_v(s)<T^\circlearrowleft_u(s),\; T_v(s)<s+\delta s\,\big|\,\rv{\xi}=\xi,X(s)=u\big) \\ \nonumber
&\leq \xPr_u^{(s+\delta s)}(T_v<T^\circlearrowleft_u,\; T_v\leq\delta s) \\ \nonumber
&\leq \xPr_u^{(s+\delta s)}(T_v<T^\circlearrowleft_u),
\end{align}
where $\xPr^{(s)}$ denotes the probability law for the time-homogeneous version of the process having parameters $\lambda_U(s)$ and $\lambda_V(s)$, and $\xPr^{(s)}_u$ stands for $\xPr^{(s)}(\cdot\,|\,X(0)=u)$. For the other term, we can write
\begin{align}
\MoveEqLeft
\xPr\big(s+\delta s<T_v(s)<T^\circlearrowleft_u(s)\,\big|\,\rv{\xi}=\xi,X(s)=u\big) \\
&= \xPr\big(T_v(s)<T^\circlearrowleft_u(s)\,\big|\,\rv{\xi}=\xi,X(s)=u\big) \nonumber\\
& \quad \times \xPr\big(T_v(s)>s+\delta s\,\big|\,\rv{\xi}=\xi,\,X(s)=u,\,\; T_v(s)<T^\circlearrowleft_u(s)\big) \nonumber \\
&= \xPr\big(B(s)=1\,\big|\,\rv{\xi}=\xi\big)
\,\,\xPr\big(T_v(s)>s+\delta s\,\big|\,\rv{\xi}=\xi,\,X(s)=u,\,\; T_v(s)<T^\circlearrowleft_u(s)\big). \nonumber
\end{align}
We get a suitable bound if
\begin{align}
\xPr\big(T_v(s)>s+\delta s\,\big|\,\rv{\xi}=\xi,X(s)=u\,\; T_v(s)<T^\circlearrowleft_u(s)\big) 
&= \smallo(1)
\end{align}
as $\lambda\to\infty$, as long as $\xi$ has at least $k$ ticks in $(s,s+\delta s)$. But if the condition $\xi\big((s,s+\delta s)\big)\geq k$ is satisfied, then we can estimate
\begin{align}
\MoveEqLeft
\xPr\big(T_v(s)>s+\delta s\,\big|\,\rv{\xi}=\xi,\,X(s)=u,\,\; T_v(s)<T^\circlearrowleft_u(s)\big) \\ \nonumber
&\leq \xPr\big(\delta\hat{T}_v(s)>k\,\big|\,
\rv{\xi}=\xi,\,X(s)=u,\,\;\delta\hat{T}_v(s)<\delta\hat{T}^\circlearrowleft_u(s)\big),
\end{align}
which is assumed to be $\smallo(1)$ as $\lambda\to\infty$.
\end{proof}


\subsection{Simplified condition for the upper bound}
\label{S3.2}

In order to apply Proposition~\ref{prop:abstract:upper-bound} effectively, we need a bound $\xExp\big[(1-E')^{-(\tilde{L}(s)-1)}\,\big|\,\rv{\xi}\big] \leq 1+E'\smallo(1)$ as $\lambda\to\infty$, whenever $\rv{\xi}$ is in a highly probable set $\Xi_t$. Recall that in our setting, $\check{\varepsilon}(s)$ can be chosen to be the success probability of the trial at time $s$ if we freeze the parameters $\lambda_U$ and $\lambda_V$ at time $s$, and $E'$ can be chosen to be $\check{\varepsilon}(t)$. So, we have $E'=\smallo(1)$ as $\lambda\to\infty$. The variable $\tilde{L}(s)$ is distributed as the discrete return time to $u$ of the (time-inhomogeneous) Markov chain conditioned on the event $T^\circlearrowleft_u<T_v$. If we would not have the time-inhomogeneity, then our study of the time-homogeneous setting in \cite{HolNarTaa16} would imply that $\xExp_u[\tilde{L}(s)]=1+\smallo(1)$ as $\lambda\to\infty$. We expect that time-inhomogeneity does not really affect this estimate and that $\tilde{L}(s)$ remains close to $1$ with high probability and in expectation, even if conditioned on $\rv{\xi}$. In the following proposition, $\varepsilon$ can be chosen to be $E'$ from Proposition~\ref{prop:abstract:upper-bound}.

\begin{proposition}[\textbf{Simplified upper bound}]
\label{prop:upper-bound:simplified}
Let $t=M(\lambda)\tau$, where $M(\lambda)\to\infty$ as $\lambda\to\infty$ and $\tau\geq 0$ is a constant. Let $\varepsilon>0$ (depending on $\lambda$) be such that $\varepsilon=\smallo(1)$ as $\lambda\to\infty$, and set $r\isdef\frac{1}{1-\varepsilon}$. Suppose that $C\geq 1$ is an integer (possibly depending on $\lambda$) and $\Xi_t=\Xi_t(\lambda)$ is a measurable set such that
\begin{enumerate}[label={\rm(\alph*)}]
\item $\varepsilon\, C=\smallo(1)$ as $\lambda\to\infty$,
\item $\xExp_u[\tilde{L}(s)\indicator{\tilde{L}(s)\leq C+1}\,|\,\rv{\xi}=\xi]=1+\smallo(1)$ as $\lambda\to\infty$,
\item $C\xPr_u(\tilde{L}(s)>C+1\,|\,\rv{\xi}=\xi)=\smallo(1)$ as $\lambda\to\infty$,
\item $\sup_{x\notin\{u,v\}}\xPr_x(\tilde{L}(s)>C\,|\,\rv{\xi}=\xi)=\smallo(1)$ as $\lambda\to\infty$,
\end{enumerate}
uniformly in $\xi\in\Xi_t$ and $s\in\bar{\xi}\cap[0,t]$. Then $\xExp_u[r^{\tilde{L}(s)-1}\,|\,\rv{\xi}=\xi]\leq 1+\varepsilon\,\smallo(1)$ as $\lambda\to\infty$, uniformly in $\xi\in\Xi_t$ and $s\in\bar{\xi}\cap[0,t]$.
\end{proposition}

\begin{proof}
Throughout the proof we assume that $\rv{\xi}\in\Xi_t$. Abbreviate $\Delta(s)\isdef\tilde{L}(s)-1$. The idea is to break down the possibilities according to whether $\Delta$ is small or large:
\begin{itemize}
\item 
For $\Delta$ small, we have $r^\Delta=\big(\frac{1}{1-\varepsilon}\big)^\Delta\approx 1+\varepsilon\Delta$, which on average is $1+\varepsilon\,\smallo(1)$.
\item 
For $\Delta$ large, the exponential tail of the distribution of $\Delta$ cancels the exponential $r^\Delta$.
\end{itemize}
To make this rigorous, we write
\begin{equation}
\xExp_u[r^{\Delta(s)}\,|\,\rv{\xi}] = \xExp_u[r^{\Delta(s)}\,\indicator{\Delta(s)\leq C}\,|\,\rv{\xi}] 
+ \xExp_u[r^{\Delta(s)}\,\indicator{\Delta(s)> C}\,|\,\rv{\xi}]
\end{equation}
and estimate each term separately.
	
\begin{lemma}[\textbf{$\Delta$ small}]
\label{lem:small}
$\xExp_u[r^{\Delta(s)}\,\indicator{\Delta(s)\leq C}\,|\,\rv{\xi}]\leq \xPr_u\big(\Delta(s)\leq C\,\big|\,\rv{\xi}\big) + \varepsilon\,\smallo(1)$ as $\lambda\to\infty$.
\end{lemma}
	
\begin{proof}[Proof of Lemma~\ref{lem:small}]
For $x\geq-1$ and $k\geq 1$, we have $(1+x)^k\geq 1+kx$. Therefore
\begin{equation}
(1-\varepsilon)^{-\Delta(s)} \leq \frac{1}{1-\varepsilon\Delta(s)} 
= 1 + \varepsilon\Delta(s)[1 + \smallo(1)].
\end{equation}
So,
\begin{equation}		
\begin{aligned}
\xExp_u[r^{\Delta(s)}\,\indicator{\Delta(s)\leq C}\,|\,\rv{\xi}] 
&\leq \sum_{\ell=0}^C \xPr_u\big(\Delta(s)=\ell\,\big|\,\rv{\xi}\big)
\big(1+\varepsilon\,\ell\,[1+\smallo(1)]\big) \\
&\leq \xPr_u\big(\Delta(s)\leq C\,\big|\,\rv{\xi}\big)
+ \varepsilon\xExp_u[\Delta(s)\indicator{\Delta(s)\leq C}\,|\,\rv{\xi}]\,[1+\smallo(1)] \\
&= \xPr_u\big(\Delta(s)\leq C\,\big|\,\rv{\xi}\big) + \varepsilon\,\smallo(1), \qquad \lambda\to\infty.
\end{aligned}
\end{equation}
\end{proof}

\begin{lemma}[\textbf{$\Delta$ large}]
\label{lem:large}
$\xExp_u[r^{\Delta(s)}\,\indicator{\Delta(s)> C}\,|\,\rv{\xi}]\leq\xPr_u\big(\Delta(s)> C\,\big|\,\rv{\xi}\big) 
+ \varepsilon\,\smallo(1)$ as $\lambda\to\infty$.
\end{lemma}

\begin{proof}[Proof of Lemma~\ref{lem:large}]
We start with
\begin{equation}
\xExp_u[r^{\Delta(s)}\,\indicator{\Delta(s)> C}\,|\,\rv{\xi}] 
= \sum_{\ell>C} \xPr_u\big(\Delta(s)=\ell\,\big|\,\rv{\xi}\big)\,r^\ell.
\end{equation}
Writing $r^\ell$ telescopically as
\begin{equation}
r^\ell = r^C + \sum_{k=C}^{\ell-1}(r^{k+1}-r^k)
= r^C + (r-1)\sum_{k=C}^{\ell-1}r^k,
\end{equation}
we get
\begin{equation}
\begin{aligned}
\xExp_u[r^{\Delta(s)}\,\indicator{\Delta(s)> C}\,|\,\rv{\xi}] 
&= \xPr_u\big(\Delta(s)>C\,\big|\,\rv{\xi}\big)\,r^C 
+ \sum_{\ell>C}\xPr_u\big(\Delta(s)=\ell\,\big|\,\rv{\xi}\big)(r-1)\sum_{k=C}^{\ell-1}r^k \\
&= \xPr_u\big(\Delta(s)>C\,\big|\,\rv{\xi}\big)\,r^C 
+ (r-1)\sum_{k\geq C}r^k\sum_{\ell>k}\xPr_u\big(\Delta(s)=\ell\,\big|\,\rv{\xi}\big) \\
&= \xPr_u\big(\Delta(s)>C\,\big|\,\rv{\xi}\big)\,r^C + (r-1)
\sum_{k\geq C}\xPr_u\big(\Delta(s)>k\,\big|\,\rv{\xi}\big)\,r^k.
\end{aligned}
\end{equation}
For the first term, by the argument for the previous claim, $r^C=(1-\varepsilon)^{-C}\leq 1/(1-\varepsilon C)=1+\varepsilon C[1+\smallo(1)]$. Therefore
\begin{equation}		
\begin{aligned}
\xPr_u\big(\Delta(s)>C\,\big|\,\rv{\xi}\big)\,r^C 
&\leq \xPr_u\big(\Delta(s)>C\,\big|\,\rv{\xi}\big)
+ \varepsilon\,C\xPr_u\big(\Delta(s)>C\,\big|\,\rv{\xi}\big) [1+\smallo(1)] \\
&= \xPr_u\big(\Delta(s)>C\,\big|\,\rv{\xi}\big) + \varepsilon\,\smallo(1), 
\qquad \lambda\to\infty.
\end{aligned}
\end{equation}
Since $r-1=\varepsilon[1+\smallo(1)]$, it remains to show that $\sum_{k\geq C}\xPr_u(\Delta(s)>k\,\big|\,\rv{\xi})r^k=\smallo(1)$. To this end, note that
\begin{align}
\delta &\isdef \sup_{\xi\in\Xi_t}\sup_{s\in\bar{\xi}\cap[0,t]}\sup_{x\notin\{u,v\}}
\xPr_x\big(\tilde{L}(s)>C\,\big|\,\rv{\xi}=\xi\big) = \smallo(1),
\qquad \lambda\to\infty.
\end{align}
Slicing time into intervals of length $C$ and using the Markov property, we get
\begin{equation}
\xPr_u\big(\Delta(s)>C+iC+j\,\big|\,\rv{\xi}\big) \leq \xPr_u\big(\Delta(s)>C\,\big|\,\rv{\xi}\big)\,\delta^i
\end{equation}
for every $i,j\geq 0$. Therefore
\begin{equation}
\begin{aligned}
\sum_{k\geq C}\xPr_u\big(\Delta(s)>k\,\big|\,\rv{\xi}\big)r^k 
&= \sum_{i\in\NN_0} \sum_{j=0}^{C-1} \xPr_u\big(\Delta(s)>C+iC+j\,\big|\,\rv{\xi}\big)\,r^{C+iC+j} \\
&\leq \xPr_u\big(\Delta(s)>C\,\big|\,\rv{\xi}\big)\,r^C \sum_{i\in\NN_0}\delta^i r^{iC} \sum_{j=0}^{C-1}r^j \\
&\leq C\xPr_u\big(\Delta(s)>C\,\big|\,\rv{\xi}\big)\,r^{2C} \sum_{i\in\NN_0}\delta^i r^{i C}.
\end{aligned}
\end{equation}
Since $r^C=1+\varepsilon C[1+\smallo(1)]=1+\smallo(1)$ and $\delta=\smallo(1)$, it follows that 
\begin{equation}
\sum_{i\in\NN_0} \delta^i r^{i C} = \frac{1}{1-\delta r^C} = 1+\smallo(1), \qquad \lambda\to\infty.
\end{equation}
We also have $r^{2C}=1+\smallo(1)$. Finally, recall that
\begin{align}
C\xPr_u\big(\Delta(s)>C\,\big|\,\rv{\xi}\big)
&= \smallo(1), \qquad \lambda\to\infty.
\end{align}
Altogether, we find that
\begin{equation}
\sum_{k\geq C}\xPr_u\big(\Delta(s)>k\,\big|\,\rv{\xi}\big)r^k = \smallo(1), \qquad \lambda\to\infty.
\end{equation}
\end{proof}
Lemmas \ref{lem:small}--\ref{lem:large} complete the proof of Proposition \ref{prop:upper-bound:simplified}.
\end{proof}


\subsection{Summary and conditions}
\label{S3.3}

In this section we put together the results obtained so far to prove the middle case of identity~\eqref{eq:main-result} subject to the validity of certain hypotheses. These hypotheses will be evaluated in Section~\ref{S4}, and lead to the proof of Theorem~\ref{thm:main}.

\paragraph*{Short-time regularity conditions and choice of the parameters}

Recall the notation $\delta\hat{T}_v(s)$ for the number of ticks of $\rv{\xi}$ from time $s$ until the first hitting time of $v$, i.e., $\delta\hat{T}_v(s)\isdef\rv{\xi}((s,T_v(s)])$, and similarly $\delta\hat{T}^\circlearrowleft_u(s)\isdef\rv{\xi}((s,T^\circlearrowleft_u(s)])$. With a similar notation, the random variable $L(s)$ introduced earlier is the same as $\delta\hat{T}^\circlearrowleft_{\{u,v\}}(s)$, and its modified version satisfies
\begin{align}
\xPr\big(\tilde{L}(s)\in\cdot\,\big|\,\rv{\xi}\big) 
&= \xPr_u\big(\delta\hat{T}^\circlearrowleft_u\in\cdot\,\big|\,\rv{\xi},\,\delta\hat{T}^\circlearrowleft_u(s)
<\delta\hat{T}_v(s)\big).
\end{align}
Let
\begin{align}
\check{\varepsilon}(s) 
&\isdef \xPr_u^{(s)}(T_v<T^\circlearrowleft_u)
\end{align}
be the probability that the time-homogeneous Markov chain with parameters $\lambda_U(s)$ and $\lambda_V(s)$ starting from $u$ hits $v$ before returning to $u$.

As before, we consider the time scaling $t=M(\lambda)\tau$, where $M(\lambda)\to\infty$ as $\lambda\to\infty$ and $\tau\geq 0$. Combining Propositions~\ref{prop:abstract:lower-bound}, \ref{prop:abstract:upper-bound}, \ref{prop:lower-bound:typical} and \ref{prop:upper-bound:simplified}, we see that it remains to verify the following conditions for suitable choices of the parameters $C=C(\lambda)\in\ZZ^+$, $k=k(\lambda)\in\ZZ^+$, $m=m(\lambda)\in\ZZ^+$ and $\delta s=\delta s(\lambda)\in\RR^+$ and a measurable set $\Xi_{M\tau}$ satisfying $\xPr(\rv{\xi}\in\Xi_{M\tau})\to 1$ as $\lambda\to\infty$:

\medskip

\noindent
\uline{\textsl{Short-time regularity conditions:}}\nopagebreak
\begin{enumerate}[label={\protect\itemboxed{\rm\Roman*}}] 
\begin{samepage}	
\item 
\label{item:task:truncated-mean}
$\xExp\big[\delta\hat{T}^\circlearrowleft_u(s)\indicator{\delta\hat{T}^\circlearrowleft_u(s)\leq C+1}\,\big|\,\rv{\xi}=\xi,\,
X(s)=u,\,\delta\hat{T}^\circlearrowleft_u(s)<\delta\hat{T}_v(s)\big] = 1+\smallo(1)$ as $\lambda\to\infty$,
uniformly in $\xi\in\Xi_{M\tau}$ and $s\in\bar{\xi}\cap[0,M\tau]$.
\end{samepage}
\item 
\label{item:task:conditional-return}
$C\xPr\big(\delta\hat{T}^\circlearrowleft_u(s)>C+1\,\big|\,\rv{\xi}=\xi,\,X(s)=u,\,\delta\hat{T}^\circlearrowleft_u(s)<\delta\hat{T}_v(s)\big)
= \smallo(1)$ as $\lambda\to\infty$,\\
uniformly in $\xi\in\Xi_{M\tau}$ and $s\in\bar{\xi}\cap[0,M\tau]$.
\item 
\label{item:task:conditional-u}
$\sup_{x\notin\{u,v\}}\xPr\big(\delta\hat{T}_u(s)>C\,\big|\,\rv{\xi}=\xi,\,X(s)=x,\,\delta\hat{T}_u(s)<\delta\hat{T}_v(s)\big)
= \smallo(1)$ as $\lambda\to\infty$,\\
uniformly in $\xi\in\Xi_{M\tau}$ and $s\in\bar{\xi}\cap[0,M\tau]$.
\item 
\label{item:task:conditional-v}
$\xPr\big(\delta\hat{T}_v(s)>k\,\big|\,
\rv{\xi}=\xi,X(s)=u,\;\delta\hat{T}_v(s)<\delta\hat{T}^\circlearrowleft_u(s)\big) = \smallo(1)$ as $\lambda\to\infty$,\\ 
uniformly in $\xi$ and $s\in\bar{\xi}\cap[0,M\tau]$ satisfying $\xi\big((s,s+\delta s)\big)\geq k$,
\item 
\label{item:task:last-excursion}
For every sequence $(\lambda_n)_{n\in\NN}$ going to infinity, there exists a subsequence $(\lambda_{n(i)})_{i\in\NN}$ such that, for all but at most countably many values $\tau\in[0,\infty)$, $\xPr_u(S\leq M\tau<T_v)=\smallo(1)$ when $\lambda\isdef\lambda_{n(i)}$ and $i\to\infty$.
\end{enumerate}

\smallskip

\noindent
\uline{\textsl{Choice of the parameters:}}
\begin{enumerate}[label={\protect\itemcircled{\rm\roman*}}] 
\begin{samepage}
\item 
\label{item:task:parameter:C}
$\check{\varepsilon}(t) C=\smallo(1)$.
\end{samepage}
\item 
\label{item:task:parameter:delta-s:gamma}
$\gamma(s+\delta s)=\gamma(s)[1+\smallo(1)]$ uniformly in $s\in[0,M\tau]$.
\item 
\label{item:task:parameter:delta-s:small}
$\gamma(M\tau)\check{\varepsilon}(M\tau+\delta s)\delta s=\smallo(1)\int_0^{M\tau} 
\check{\varepsilon}(s)\gamma(s)\,\dd s$.
\item 
\label{item:task:parameter:k}
$\delta_{M\tau}(m)\isdef\xPr(N_{M\tau}>m)=\smallo(1)$, where $N_t$ is the number of points $s\in\bar{\rv{\xi}}\cap[0,t]$
for which $\rv{\xi}\big((s,s+\delta s)\big)<k$.
\item 
\label{item:task:parameter:E-vs-m}
$\big(1-\check{\varepsilon}(t)\big)^m=1-\smallo(1)$.
\end{enumerate}

\paragraph*{Summary}

The following proposition summarizes our results so far.

\begin{proposition}[\textbf{Law of crossover time}]
\label{prop:remaining-tasks}
Let $M(\lambda)$ be a time scale with $M(\lambda)\to\infty$ as $\lambda\to\infty$.  Suppose that $C,k,m\in\ZZ^+$, $\delta s\in\RR^+$ and a measurable set $\Xi_{M\tau}$ with $\xPr(\rv{\xi}\in\Xi_{M\tau})=\smallo(1)$ as $\lambda\to\infty$ can be chosen (each possibly depending on $\lambda$) such that the above Conditions~\ref{item:task:parameter:C}--\ref{item:task:parameter:E-vs-m} and \ref{item:task:truncated-mean}--\ref{item:task:last-excursion} are satisfied. Then, for every sequence $(\lambda_n)_{n\in\NN}$ going to infinity, there exists a subsequence $(\lambda_{n(i)})_{i\in\NN}$ such that, for $\lambda\isdef\lambda_{n(i)}$ and all but at most countably many values $\tau\in[0,\infty)$ satisfying $\check{\varepsilon}(M\tau)\int_0^{M\tau}\gamma(s)\dd s = \bigo(1)$,
\begin{equation}
\label{eq:sandwich}
\begin{aligned}
\xPr_u(T_v>M\tau) 
&\leq [1-\smallo(1)]\,\ee^{-\textstyle{\int_0^{M\tau}\check{\varepsilon}(s)\gamma(s)\dd s}}
+ \smallo(1) \;, \quad &i\to\infty\;,\\[0.2cm]
\xPr_u(T_v>M\tau) 
 &\geq [1-\smallo(1)]\,\ee^{-\textstyle{\int_0^{M\tau}\check{\varepsilon}(s)\gamma(s)\dd s}}\;,
&i\to\infty.
\end{aligned}
\end{equation}
\end{proposition}

\begin{proof}
We establish the upper and the lower bounds separately.  The restriction to a subsequence $(\lambda_{n(i)})_{i\in\NN}$ and a co-countable set of values $\tau\in[0,\infty)$ is needed only for the upper bound, which relies on~\ref{item:task:last-excursion}.  The lower bound holds for every $\tau\in[0,\infty)$ as $\lambda\to\infty$.
	
\smallskip
\noindent\uline{\textsl{Upper bound}}.
By~\ref{item:task:last-excursion}, for every sequence $(\lambda_n)_{n\in\NN}$ going to infinity, there exists a subsequence $(\lambda_{n(i)})_{i\in\NN}$ such that as $i\to\infty$, for all but countably many values $\tau\in[0,\infty)$,
\begin{align}
\xPr_u(T_v>M\tau) &= \xPr(S>M\tau) + \xPr_u(S\leq M\tau<T_v)
= \xPr(S>M\tau) + \smallo(1) \;.
\end{align}
To bound $\xPr(S>M\tau)$, we apply Propositions~\ref{prop:abstract:upper-bound} and~\ref{prop:upper-bound:simplified}. We choose $E'\isdef\check{\varepsilon}(t)$. The condition $\check{\varepsilon}(s)\leq E'$ will then be satisfied for each $0\leq s\leq t$ by monotonicity (see the paragraph at the end of Section~\ref{sec:intro:formulation}). By~\ref{item:task:parameter:C}, \ref{item:task:truncated-mean}, \ref{item:task:conditional-return} and~\ref{item:task:conditional-u}, the conditions of Proposition~\ref{prop:upper-bound:simplified} are satisfied with $\varepsilon\isdef\check{\varepsilon}(t)$, and thus $\xExp_u[r^{\tilde{L}(s)-1}\,|\,\rv{\xi}=\xi]\leq \rho(s)\isdef 1+\check{\varepsilon}(t)\,\smallo(1)$ uniformly for $\xi\in\Xi_t$ and $s\in\bar{\xi}\cap[0,t]$, where $r\isdef\frac{1}{1-E'}$. Therefore, the conditions of Proposition~\ref{prop:abstract:upper-bound} are satisfied. Observe that
\begin{align}
\hat{K} 
&= \rho(0)\big(1-\check{\varepsilon}(0)\big)
\ee^{\,\textstyle{\int_0^t\big(\rho(s)-1\big)\gamma(s)\dd s}} \\ \nonumber
&= [1+\check{\varepsilon}(t)\,\smallo(1)][1-\smallo(1)]
\ee^{\,\smallo(1)\,\check{\varepsilon}(t)\textstyle{\int_0^t\gamma(s)\dd s}} \\ \nonumber
&= 1-\smallo(1)\,,\qquad\lambda\to\infty\;,
\end{align}
where the last equality uses the hypothesis $\check{\varepsilon}(t)\int_0^{t}\gamma(s)\dd s = \bigo(1)$.
	
\smallskip
\noindent\uline{\textsl{Lower bound}}.
Recall that $T_v=S+\delta T(S)$, and hence $\xPr_u(T_v>t)\geq\xPr(S>t)$. We apply Propositions~\ref{prop:abstract:lower-bound} and~\ref{prop:lower-bound:typical}. Choose $E\isdef\check{\varepsilon}(t)$. The condition $\xPr\big(B(s)=1\,\big|\,\rv{\xi}=\xi\big)\leq E$ for every $\xi$ and $s\in\bar{\xi}\cap[0,t]$ follows from monotonicity, thanks to the assumption made in~\eqref{eq:assumption:freezing-at-the-end}.
	
By Proposition~\ref{prop:lower-bound:typical} and \ref{item:task:conditional-v},
there exists a function $v(\lambda)=\smallo(1)$ such that
\begin{align}
\xPr\big(B(s)=1\,|\,\rv{\xi}=\xi\big) &\leq
\xPr_u^{(s+\delta s)}(T_v<T_u^\circlearrowleft)[1+v(\lambda)] =
\check{\varepsilon}(s+\delta s)[1+v(\lambda)]
\end{align}
for every $\xi$ and $s\in\bar{\xi}\cap[0,t]$ satisfying $\xi\big((s,s+\delta s)\big)\geq k$. Thus, the conditions of Proposition~\ref{prop:abstract:lower-bound} are satisfied with $E=\check{\varepsilon}(t)$ and $\hat{\varepsilon}(s)\isdef\check{\varepsilon}(s+\delta s)[1+v(\lambda)]$, with $\Xi_t$ being the set of all point configurations $\xi$. Hence, we get
\begin{align}
\label{eq:prop:remaining-tasks:proof:lower-bound:1}
\xPr_u(T_v>t) &\geq \xPr(S>t) \\ \nonumber
&\geq \check{K}(m)\, \ee^{-\textstyle{\int_0^t \hat{\varepsilon}(s)\gamma(s)\,\dd s}}
- \smash{\overset{0}{\cancel{\xPr(\rv{\xi}\notin\Xi_t)}}} \\ \nonumber
&= \check{K}(m)\, \ee^{-[1+v(\lambda)]\textstyle{\int_0^t \check{\varepsilon}(s+\delta s)
\gamma(s)\,\dd s}} \\ \nonumber
&= \check{K}(m)[1-\smallo(1)]\, \ee^{-\textstyle{\int_{\delta s}^{t+\delta s} \check{\varepsilon}(s)
\gamma(s-\delta s)\,\dd s}} \\ \nonumber
&\geq \check{K}(m)[1-\smallo(1)]\, \ee^{-[1-\smallo(1)]\textstyle{\int_0^t \check{\varepsilon}(s)\gamma(s)\,\dd s}},
\end{align}
where
\begin{align}
\label{eq:prop:remaining-tasks:proof:lower-bound:2}
\check{K}(m) 
&\isdef \big(1-\hat{\varepsilon}(0)\big) \,
\Big[(1-\check{\varepsilon}(t))^m - \sqrt{\delta_t(m)}\,\ee^{\,\textstyle{\frac{1}{2}\int_0^t
\hat{\varepsilon}(s)^2\gamma(s)\,\dd s}}\Big] \\ \nonumber
&= [1-\smallo(1)]\Big[[1-\smallo(1)]-\smallo(1)\Big] \\ \nonumber
&= 1-\smallo(1) \;.
\end{align}
In~\eqref{eq:prop:remaining-tasks:proof:lower-bound:1}, we have used \ref{item:task:parameter:delta-s:gamma} and~\ref{item:task:parameter:delta-s:small}. The equality in~\eqref{eq:prop:remaining-tasks:proof:lower-bound:2} follows from~\ref{item:task:parameter:C}, \ref{item:task:parameter:delta-s:small}, \ref{item:task:parameter:k} and~\ref{item:task:parameter:E-vs-m}, and the fact that
\begin{align}
\int_0^t\hat{\varepsilon}(s)^2\gamma(s)\,\dd s 
&= [1+v(\lambda)]^2\int_0^t\check{\varepsilon}(s+\delta s)^2\gamma(s)\,\dd s \\ \nonumber
&\asymp \int_{\delta s}^{t+\delta s}\check{\varepsilon}(s)^2\gamma(s-\delta s)\,\dd s \\ \nonumber
&\preceq \int_{\delta s}^{t}\check{\varepsilon}(s)^2\gamma(s-\delta s)\,\dd s +
\check{\varepsilon}(t+\delta s)^2\gamma(t)\delta s \\ \nonumber
&\preceq \check{\varepsilon}(t)\int_{0}^{t}\check{\varepsilon}(s)\gamma(s)\,\dd s +
\check{\varepsilon}(t+\delta s)\gamma(t)\delta s \\ \nonumber
&\preceq \smallo(1)\int_{0}^{t}\check{\varepsilon}(s)\gamma(s)\,\dd s +
\smallo(1)\int_{0}^{t}\check{\varepsilon}(s)\gamma(s)\,\dd s \\ \nonumber
&\preceq \smallo(1)\,\check{\varepsilon}(t)\int_0^{t}\gamma(s)\dd s \\ \nonumber
&= \smallo(1) \;,
\end{align}
where we use the monotonicity of $\check{\varepsilon}(s)$.
\end{proof}

Let us point out that Proposition~\ref{prop:remaining-tasks}  is valid for any choice of the underlying bipartite graph $G$, and any choice of the functions $g_U(\cdot)$ and $g_V(\cdot)$ satisfying $g_V(x)\succ g_V(x)\succ 1$ as $x\to\infty$.


\section{Proof of the conditions}
\label{S4}

In this section we establish \ref{item:task:truncated-mean}--\ref{item:task:last-excursion} for suitable choices of the parameters $C$, $k$, $m$ and $\delta s$ satisfying \ref{item:task:parameter:C}--\ref{item:task:parameter:E-vs-m}. In Section~\ref{sec:tasks:parameters} we simplify Conditions~\ref{item:task:parameter:C}--\ref{item:task:parameter:E-vs-m}, obtaining more explicit conditions for $C$, $k$, $m$ and $\delta s$. In Section~\ref{sec:reghom} we explain why~\ref{item:task:truncated-mean}--\ref{item:task:last-excursion} are expected to be true for suitable choices of $C$, $k$, $m$ and $\delta s$ by examining similar statements~\ref{item:task:truncated-mean:time-homogeneous}--\ref{item:task:last-excursion:time-homogeneous} in the time-homogeneous setting. In Section~\ref{sec:reginhom} we use a coupling argument to show that, with regard to these statements (which all concern short time intervals), the time-inhomogeneous setting behaves similarly as the time-homogeneous setting. In Section~\ref{sec:proofmainth} we put the pieces together and prove Theorem~\ref{thm:main}. 


\subsection{Choice of the parameters}
\label{sec:tasks:parameters}

In this section, we specialize to the particular form of the functions $g_U(\cdot)$ and $g_V(\cdot)$ chosen in Section~\ref{sec:intro:theorem}. The choice of the underlying bipartite graph remains completely arbitrary.

Let us start by recalling the choices
\begin{align}
\label{eq:rates:choice:recall}
\lambda_U(s) 
&= \begin{cases}
(c_U\lambda - \mu_U s)^{\beta_U} & \text{if $s<\frac{c_U}{\mu_U}\lambda$,} \\
0 & \text{otherwise,}
\end{cases}
&\lambda_V(s) 
&= (c_V\lambda + \mu_V s)^{\beta_V},
\end{align}
for $s\leq M\tau$, where $\beta_V>\beta_U>0$. Note that when $s\geq \frac{c_U}{\mu_U}\lambda$, we have $\lambda_U(s)=0$ and $\lambda_V(s) \to\infty$ as $\lambda\to\infty$.  If the crossover has not occurred by time $s = \frac{c_U}{\mu_U}\lambda$, then it will happen in a time of order $\bigo(1)$.  Namely, it will take an exponential time with rate $1$ for each vertex in $U$ to become inactive, independently for different vertices, and once $U$ is completely inactive, the complete activation of $V$ happens in time $\smallo(1)$. Let us therefore focus on the case $s<\frac{c_U}{\mu_U}\lambda$.

\begin{lemma}[\textbf{Choice of parameters}]
\label{lem:tasks:parameters:simplified}
Let $s\leq M\tau$, and let $\lambda_U(s)$ and $\lambda_V(s)$ be as in \eqref{eq:rates:choice:recall}. Consider the time scaling $s=M\sigma$ with $M=M(\lambda)>0$ and $\sigma\in[0,\infty)$, and suppose that either $1\preceq M\prec\lambda$, or $M\asymp\lambda$ and $M\sigma<\frac{c_U}{\mu_U} \lambda$. Then, Conditions~\ref{item:task:parameter:C}--\ref{item:task:parameter:E-vs-m} are met when
\begin{align}
\label{eq:task:parameter:simplified}
C &= \frac{\smallo(1)}{\check{\varepsilon}(M)},
&m &= \frac{\smallo(1)}{\check{\varepsilon}(M)},
&\delta s &= \smallo(M)\;,
&k &\leq \tfrac{1}{2}\left[\inf_{u}\gamma(u)\right] \delta s,
&m\,\delta s &\succ M \;,
\end{align}
as $\lambda\to\infty$.
\end{lemma}

\begin{proof}
We consider the two scaling regimes separately.

\smallskip

\noindent\textsl{\uline{Regime $1\preceq M\prec\lambda$}.}\quad
We begin with some observations. In this case
\begin{align}
\lambda_U(M\sigma) &= c_U^{\beta_U}\lambda^{\beta_U}[1+\smallo(1)],
&\lambda_V(M\sigma) &= c_V^{\beta_V}\lambda^{\beta_V}[1+\smallo(1)],
\end{align}
for every $\sigma\geq 0$. This means that the orders of magnitude of $\lambda_U(M\sigma)$ and $\lambda_V(M\sigma)$ (up to their pre-factors) do not change with the scaled time $\sigma$. Clearly, for fixed $\tau\geq 0$, the $\smallo(1)$ terms in the above asymptotics are uniform in $\sigma\in[0,\tau]$. It follows that
\begin{align}
\label{eq:task:parameters:gamma:M-small}
\gamma(M\sigma) 
&= \big(1+\lambda_U(M\sigma)\big)\abs{U} + \big(1+\lambda_V(M\sigma)\big)\abs{V}
= \abs{V}c_V^{\beta_V}\lambda^{\beta_V}[1+\smallo(1)],
\end{align}
where the $\smallo(1)$ term is again uniform in $\sigma\in[0,\tau]$. Recall that 
\begin{align}
\check{\varepsilon}(s) 
&= \xPr^{(s)}_u(T_v<T^\circlearrowleft_u) = \frac{1}{\pi^{(s)}(u)\effR[(s)]{u}{v}},
\end{align}
where $\pi^{(s)}$ is the stationary probability of the Markov chain with parameters $\lambda_U(s)$ and $\lambda_V(s)$, and $\effR[(s)]{u}{v}$ is the effective resistance between $u$ and $v$ in the same Markov chain (see~\cite[Proposition 9.5]{LevPerWil08}). Note that both $\pi^{(s)}(u)$ and $\effR[(s)]{u}{v}$ are rational functions of $\lambda_U(s)$ and $\lambda_V(s)$. Namely, the stationary distribution $\pi^{(s)}$ is the solution of a system of linear equations whose coefficients are rational in $\lambda_U(s)$ and $\lambda_V(s)$. Likewise, the effective conductance $1/\effR[(s)]{u}{v}$ is the strength of the current flow from $u$ to $v$ when we put a unit battery between $u$ and $v$, and hence is a linear combination of the voltage values with coefficients that are rational in $\lambda_U(s)$ and $\lambda_V(s)$.  The voltage associated with the unit battery between $u$ and $v$ (i.e., a harmonic function with boundary conditions $0$ and $1$ at $u$ and $v$, respectively) is itself the solution of a linear system of equations whose coefficients are rational in $\lambda_U(s)$ and $\lambda_V(s)$. It follows that
\begin{align}
\label{eq:task:parameters:eps:M-small}
\check{\varepsilon}(M\sigma) &= \check{\varepsilon}(0)[1+\smallo(1)],
\end{align}
where the $\smallo(1)$ term is uniform in $\sigma\in[0,\tau]$.

Next we discuss the choice of parameters $C$, $k$, $m$ and $\delta s$ in order for Conditions~\ref{item:task:parameter:C}--\ref{item:task:parameter:E-vs-m} to be fulfilled. In order to satisfy~\ref{item:task:parameter:C} and~\ref{item:task:parameter:E-vs-m}, we choose $C=\frac{\smallo(1)}{\check{\varepsilon}(M)}$ and $m=\frac{\smallo(1)}{\check{\varepsilon}(M)}$ as $\lambda\to\infty$. In light of the above discussion, Condition~\ref{item:task:parameter:delta-s:gamma} is automatically satisfied because of~\eqref{eq:task:parameters:gamma:M-small}. By~\eqref{eq:task:parameters:gamma:M-small} and~\eqref{eq:task:parameters:eps:M-small}, Condition~\ref{item:task:parameter:delta-s:small} is satisfied as long as $\delta s = \smallo(M)$. To simplify Condition~\ref{item:task:parameter:k}, we use the following lemma.

\begin{lemma}
\label{lem:lower-bound:poisson}
Let $N_t$ denote the number of points $s\in\bar{\rv{\xi}}\cap[0,t]$ such that $\bar{\rv{\xi}}\big((s,s+\delta s)\big)<k$, and $\delta_t(m)\isdef\xPr(N_t>m)$. Then
\begin{align}
\delta_{M\tau}(m)
&\preceq \frac{M}{m\,\delta s},
\qquad\lambda\to\infty,
\end{align}
provided $k\leq\frac{1}{2}[{\displaystyle\inf_u\gamma(u)}]\delta s$.
\end{lemma}

\begin{proof}[Proof of Lemma~\ref{lem:lower-bound:poisson}]
We start by writing
\begin{align}
\xExp[N_t] 
&= \xExp\Bigg[\sum_{s\in\bar{\rv{\xi}}\cap[0,t]}\indicator{\bar{\rv{\xi}}((s,s+\delta s))<k}\Bigg] \\
&= \nonumber
\xPr\Big(\rv{\xi}\big((0,\delta s)\big)<k\Big) 
+ \xExp\Bigg[\sum_{s\in\rv{\xi}\cap[0,t]}\indicator{\bar{\rv{\xi}}((s,s+\delta s))<k}\Bigg] \\
&= \nonumber
\xPr\Big(\rv{\xi}\big((0,\delta s)\big)<k\Big) 
+ \int_0^t\xPr\Big(\rv{\xi}\big((r,r+\delta s)\big)<k\Big)\gamma(r)\,\dd r,
\end{align}
where the last equality follows from the Campbell-Mecke formula (see~\cite[Theorems~1.11 and~1.13]{BacBla09I}).
Note that $W_r\isdef\rv{\xi}\big((r,r+\delta s)\big)$ is a Poisson random variable with parameter $\int_r^{r+\delta s}\gamma(u)\dd u$. Therefore, choosing
\begin{align}
k &\leq \frac{1}{2}\left[\inf_u\gamma(u)\right]\delta s
\leq \frac{1}{2}\int_r^{r+\delta s}\gamma(u)\dd u
= \frac{1}{2}\xExp[W_r],
\end{align}
we can use the Chebyshev inequality to get
\begin{align}
\label{eq:lower-bound:poisson:chebyshev}
\xPr(W_r<k) 
&\leq \frac{\xVar[W_r]}{(\xExp[W_r]-k)^2}
= \frac{\int_r^{r+\delta s}\gamma(u)\dd u}{\big(\int_r^{r+\delta s}\gamma(u)\dd u - k\big)^2} \\
&\leq \nonumber
\frac{\gamma(r+\delta s)\delta s}{\big(\gamma(r)\delta s - \frac{1}{2}[\inf_u\gamma(u)]\delta s\big)^2}[1+\smallo(1)] \\
&\leq \nonumber
\frac{\gamma(r)\delta s}{\frac{1}{4}\big(\gamma(r)\delta s\big)^2}[1+\smallo(1)]
\leq \frac{4}{\gamma(r)\delta s}[1+\smallo(1)].
\end{align}
Therefore, by the Markov inequality,
\begin{align}
\delta_t(m) 
&\leq \frac{\xExp[N_t]}{m}
\leq
\frac{1}{m}\bigg[\frac{\bigo(1)}{[\inf_u\gamma(u)]\delta s} + \int_0^t\frac{\bigo(1)}{\gamma(r)\delta s}\gamma(r)\dd r\bigg] \\
&\leq \nonumber
\frac{\bigo(1)}{m\,\delta s}\bigg[\frac{1}{\inf_u\gamma(u)} + t\bigg],
\end{align}
which proves the claim because $\inf_u\gamma(u)\to\infty$ as $\lambda\to\infty$ and $t=M\tau$ with $\tau\geq 0$ a constant.	
\end{proof}

We continue with the proof of Lemma~\ref{lem:tasks:parameters:simplified}. It follows from Lemma~\ref{lem:lower-bound:poisson} that, in order to satisfy Condition~\ref{item:task:parameter:k}, we can choose $k\leq\frac{1}{2}\gamma(0)\delta s$ and make sure that $m\,\delta s\succ M$ as $\lambda\to\infty$.

\smallskip

\noindent\textsl{\uline{Regime $M\asymp\lambda$}.}\quad
Let $M\tau<\frac{c_U}{\mu_U}\lambda$. It is still the case that the orders of magnitude of $\lambda_U(M\sigma)$, $\lambda_V(M\sigma)$, $\gamma(M\sigma)$ and $\check{\varepsilon}(M\sigma)$ do not change for $\sigma\in[0,\tau]$. Hence $C=\frac{\smallo(1)}{\check{\varepsilon}(M)}$ and $m=\frac{\smallo(1)}{\check{\varepsilon}(M)}$ still guarantee Conditions~\ref{item:task:parameter:C} and~\ref{item:task:parameter:E-vs-m}. In order to satisfy Condition~\ref{item:task:parameter:delta-s:gamma}, it is enough that $\delta s=\smallo(M)$. Indeed, if $\delta s=\smallo(M)$, then
\begin{align}
\gamma(M\sigma + \delta s) 
&= \abs{V}\lambda_V(M\sigma + \delta s)[1+\smallo(1)] \\
&= \nonumber
\abs{V}(c_V\lambda+\mu_V M\sigma + \mu_V\delta s)^{\beta_V}[1+\smallo(1)] \\
&= \nonumber
\abs{V}(c_V\lambda+\mu_V M\sigma)^{\beta_V}[1+\smallo(1)] \\
&= \nonumber
\gamma(M\sigma)[1+\smallo(1)] \;.
\end{align}
Furthermore, $\delta s=\smallo(M)$ still ensures~\ref{item:task:parameter:delta-s:small}. Finally, Lemma~\ref{lem:lower-bound:poisson} remains valid when in the last two lines of~\eqref{eq:lower-bound:poisson:chebyshev} we change $[1+\smallo(1)]$ to $\bigo(1)$. Hence $k\leq\frac{1}{2}\gamma(0)\delta s$ and $m\,\delta s\succ M$ still ensure that Condition~\ref{item:task:parameter:k} is satisfied.
\end{proof}

In the proof of Lemma~\ref{lem:tasks:parameters:simplified} we made that important observation that, with the choices of parameters in~\eqref{eq:rates:choice:recall}, the orders of magnitude of $\lambda_U(s)$, $\lambda_V(s)$, $\gamma(s)$ and $\check{\varepsilon}(s)$ do not change as long as $s<\frac{c_U}{\mu_U}\lambda$.  For future reference, we state this as a separate proposition.

\begin{proposition}[\textbf{Orders of magnitude}]
\label{prop:critical-regime:orders-of-magnitude}
Consider the time scaling $s=M\sigma$ with $M=M(\lambda)>0$ and $\sigma\in[0,\tau]$.
If either $1\preceq M\prec\lambda$, or $M\asymp\lambda$ and $M\sigma<\frac{c_U}{\mu_U}\lambda$, then
\begin{align}
\lambda_U(M\sigma)&\asymp\lambda_U(0), & \lambda_V(M\sigma)&\asymp\lambda_V(0), &
\check{\varepsilon}(M\sigma)&\asymp\check{\varepsilon}(0), & \gamma(M\sigma)&\asymp\gamma(0),
\end{align}
as $\lambda\to\infty$.
\end{proposition}


\subsection{Short-time regularity in the time-homogeneous setting}
\label{sec:reghom}

We have extensively exploited the monotonicity of the hard-core model in the parameters $\lambda_U$ and $\lambda_V$. Unfortunately, this monotonicity does not provide us with any meaningful information about the conditional probabilities involved in the short-time regularity Conditions~\ref{item:task:truncated-mean}--\ref{item:task:last-excursion}. This lack of monotonicity makes the evaluation of these conditions challenging. It is helpful to first examine conditions similar to~\ref{item:task:truncated-mean}--\ref{item:task:last-excursion} in a time-homogeneous setting. In order to apply the results of~\cite{HolNarTaa16}, we will need to impose \emph{mild conditions on the isoperimetric properties of the underlying bipartite graph}.

\paragraph*{Simpler conditions}

In the time-homogeneous setting, Conditions~\ref{item:task:truncated-mean}--\ref{item:task:last-excursion} reduce to the following simpler conditions:
\begin{enumerate}[label={\protect\itemboxed{\rm\Roman*$'$}}] 
\begin{samepage}	
\item 
\label{item:task:truncated-mean:time-homogeneous}
$\xExp^{(s)}_u\big[\hat{T}^\circlearrowleft_u\indicator{\hat{T}^\circlearrowleft_u\leq C+1}\,\big|\,\hat{T}^\circlearrowleft_u<\hat{T}_v\big]
= 1+\smallo(1)$ as $\lambda\to\infty$.
\end{samepage}
\item 
\label{item:task:conditional-return:time-homogeneous}
$C\xPr^{(s)}_u\big(\hat{T}^\circlearrowleft_u>C+1\,\big|\,\hat{T}^\circlearrowleft_u<\delta\hat{T}_v\big)
= \smallo(1)$ as $\lambda\to\infty$.
\item
\label{item:task:conditional-u:time-homogeneous}
$\sup_{x\notin\{u,v\}}\xPr^{(s)}_x\big(\hat{T}_u>C\,\big|\,\hat{T}_u<\delta\hat{T}_v\big)
= \smallo(1)$ as $\lambda\to\infty$.
\item 
\label{item:task:conditional-v:time-homogeneous}
$\xPr^{(s)}_u\big(\hat{T}_v>k\,\big|\,\hat{T}_v<\hat{T}^\circlearrowleft_u\big) 
= \smallo(1)$ as $\lambda\to\infty$.
\item
\label{item:task:last-excursion:time-homogeneous}
$\xPr^{(s)}_u(S\leq M\tau<T_v)=\smallo(1)$ as $\lambda\to\infty$.
\end{enumerate}
We require these conditions to be satisfied uniformly in $s\in[0,M\tau]$. Before we proceed, let us mark a slight difference in notation compared to~\cite{HolNarTaa16}. In the present paper, $T_u^\circlearrowleft$ and $T_v$ are the continuous-time return and hitting times, while $\hat{T}_u^\circlearrowleft$ and $\hat{T}_v$ are the discrete-time versions of the return and hitting times (i.e., obtained by counting the number of ticks of the Poisson clock). The notation used in~\cite{HolNarTaa16} was the opposite, because all the analysis in that paper was based on the discrete time.

\paragraph*{Notation}
To establish~\ref{item:task:truncated-mean:time-homogeneous}--\ref{item:task:last-excursion:time-homogeneous}, we can follow different approaches. Here we use the tools developed in~\cite[Section~B.4]{HolNarTaa16} based on ideas from~\cite{ManNarOliSco04}. Let us briefly recall some relevant concepts and notation from~\cite{HolNarTaa16}. For brevity, we drop the superscript $(s)$ from $\xPr^{(s)}$, $\xExp^{(s)}$, $K^{(s)}$, $\pi^{(s)}$, etc.\ whenever there is no chance of confusion. Similarly, we write $\gamma$, $\varepsilon$, $\Gamma$, etc.\ instead of $\gamma(s)$, $\check{\varepsilon}(s)$, $\Gamma^{(s)}$, etc.

\paragraph*{Energy barriers and stability levels}
Given two distinct configurations $a,b\in\pspace{X}$, we write $a\sim b$ if $K(a,b)>0$ (or equivalently, $K(b,a)>0$). We consider a simple undirected graph on the configuration space $\pspace{X}$ in which two points $a,b\in\pspace{X}$ are connected if and only if $a\sim b$. The conductance of an edge $(a,b)$ is denoted by $c(a,b)\isdef\pi(a)K(a,b)$, and its resistance by $r(a,b)=1/c(a,b)$.  The \emph{critical resistance} between two subsets $A,B\subseteq\pspace{X}$ is defined as
\begin{align}
\Psi(A,B) &\isdef \inf_{\omega:A \pathto B} \sup_{e\in\omega}\; r(e),
\end{align}
where the infimum runs over all paths $\omega:A\pathto B$ from $A$ to $B$.  The logarithm of $\Psi(A,B)$ is often referred to as the \emph{communication height} $A$ and $B$.  It can be thought of as the (absolute) height of the smallest hill that the process needs to climb in order to go from $A$ to $B$ or vice versa.
Recall that as $\lambda\to\infty$, $\Psi(A,B)$ has the same order of magnitude as the effective resistance $\effR{A}{B}$ between $A$ and $B$.

For a state $x\in\pspace{X}$, we write
\begin{align}
J^-(x) &\isdef \{y\in\pspace{X}\colon\, \text{$\pi(y)\succ\pi(x)$ as $\lambda\to\infty$}\}
\end{align}
for the set of states $y$ that have asymptotically larger stationary probability than $x$.
The boundary of a set $A\subseteq\pspace{X}$ is defined as
\begin{align}
\partial A\isdef\{b\in\pspace{X}\setminus A\colon\, \text{$a\sim b$ for some $a\in A$}\}.
\end{align}

For $A\subseteq\pspace{X}$, define
\begin{align}
\Gamma(A) &\isdef \sup_{a\in A}\pi(a)\Psi\big(a,J^{-1}(a)\big).
\end{align}
The logarithm of $\Gamma(\{x\})$ (for $x\in\pspace{X}$) is often referred to as the \emph{stability level} of $a$, and can be thought of as the ``energy barrier'' when going from $a$ to states with higher stationary probability, or the (relative) height of the shortest hill that the process starting from $a$ needs to climb in order to reach a state with higher stationary probability.
Recall that $\xExp_x[\hat{T}_{J^-(x)}]\asymp\pi(x)\Psi\big(x,J^-(x)\big)$ for every $x$ (see \cite[Proposition B.2]{HolNarTaa16}) and, in particular, that $\sup_{x\in A}\xExp_x[\hat{T}_{A^\complement}]\preceq\Gamma(A)$ as $\lambda\to\infty$.

Let 
\begin{align}
\Gamma &\isdef \pi(u)\Psi(u,v), \\
\check{\Gamma} &\isdef
\Gamma(\pspace{X}\setminus\{u,v\}) =
\sup_{x\in\pspace{X}\setminus\{u,v\}} \pi(x)\Psi\big(x,\{u,v\}\big).
\end{align}
We know that $\varepsilon\isdef\xPr_u(\hat{T}_v<\hat{T}^\circlearrowleft_u)\asymp 1/\Gamma$. From the above remarks, we also get that $\xExp_u[\check{T}_v]\asymp\Gamma$ provided $\check{\Gamma}\preceq\Gamma$. Furthermore, if $\check{\Gamma}\prec\Gamma$ (which is the \emph{no-deep-well} property), then the transition time from $u$ to $v$ is asymptotically exponentially distributed (\cite[Corollary B.7]{HolNarTaa16}).

\begin{lemma}[\textbf{No deep well property}]
\label{lem:no-deep-well}
Subject to Assumption~\textup{\ref{keyassump}}, $\Gamma\succ\check{\Gamma}\log\gamma\succeq\log\gamma$.
\end{lemma}

\begin{proof}
Let $x\in\pspace{X}\setminus\{u,v\}$.  Note that every path $\omega:x\pathto J^-(x)$ contains a transition $y\xadd z$ where a particle is added to either $U$ or $V$, so that $\pi(y)\preceq\pi(x)\prec\pi(z)$.  Then
\begin{align}
\Psi(\omega) &\succeq r(y,z) = \frac{1}{\pi(y)K(y,z)} =
\begin{cases}
\frac{\gamma}{\pi(y)\lambda_U}	& \text{if $y\xadd[U] z$,} \\
\frac{\gamma}{\pi(y)\lambda_V}	& \text{if $y\xadd[V] z$,}
\end{cases}
\end{align}
which is $\succeq 1/\pi(y)$.
Therefore, $\pi(x)\Psi\big(x,J^-(x)\big)\succeq 1$.  It follows that $\check{\Gamma}\succeq 1$.
In the special case in which $G$ is a complete bipartite graph, every configuration $x\in\pspace{X}\setminus\{u,v\}$ has itself a missing particle, and therefore there exists a configuration $z$ such that $x\xadd z$.  It follows that, in this case, $\check{\Gamma}\asymp\gamma/\lambda_U$.

Let us next argue that $\Gamma\succ\check{\Gamma}$ under Assumption~\ref{keyassump}.
According to~\cite[Corollary~3.5]{HolNarTaa16}, $\check{\Gamma}\prec\Gamma$ as long as hypothesis~(H0) and~(H2) in~\cite{HolNarTaa16} are satisfied. Hypothesis~(H0) says that $\abs{U}<(1+\alpha)\abs{V}$, where
\begin{align}
\alpha &\isdef \lim_{\lambda\to\infty}\frac{\log\lambda_V}{\log\lambda_U}-1
= \lim_{\lambda\to\infty}\frac{\log(c_V\lambda+\mu_V s)^{\beta_V}}{\log(c_U\lambda-\mu_Us)^{\beta_U}} - 1
= \frac{\beta_V}{\beta_U} - 1.
\end{align}
On the other hand, it can be verified that if $G$ is a complete bipartite graph, then $\Gamma\asymp\gamma\lambda_U^{\abs{U}-1}$ (see~\cite[Example~2.1]{HolNarTaa16}), which is
of higher order of magnitude than~$\check{\Gamma}\asymp\gamma/\lambda_U$.

Lastly, note that $\Gamma$ and $\check{\Gamma}$ are increasing rational functions of $\lambda$, while $\log\gamma\asymp\log\lambda$. Since $\check{\Gamma}\prec\Gamma$, we in fact have $\check{\Gamma}\log\gamma\prec\Gamma$. 
\end{proof}

\paragraph*{Verificiation of Conditions~\ref{item:task:truncated-mean:time-homogeneous}--\ref{item:task:last-excursion:time-homogeneous}}

The validity of Conditions~\ref{item:task:truncated-mean:time-homogeneous} and~\ref{item:task:conditional-return:time-homogeneous} follows from part~\ref{item:tasks:time-homogeneous:paper-1:conditional-return} of the following proposition (with the help of the Markov inequality).

\begin{proposition}
\label{prop:task:truncated-mean-and-conditional-return:time-homogeneous:proof}
Subject to Assumption~{\rm \ref{keyassump}}, 
\begin{enumerate}[label={\rm(\roman*)}]
\item 
\label{item:tasks:time-homogeneous:paper-1:conditional-return}
$\xExp^{(s)}_u\big[\hat{T}^\circlearrowleft_u\,\big|\,\hat{T}^\circlearrowleft_u<\hat{T}_v\big] = 1 + \smallo(1)$
as $\lambda\to\infty$.
\item 
\label{item:tasks:time-homogeneous:paper-1:conditional-crossover}
$\xExp^{(s)}_u\big[\hat{T}_v\,\big|\,\hat{T}_v<\hat{T}^\circlearrowleft_u\big] 
= \xExp^{(s)}_u[\hat{T}_v]\,\smallo(1)$ as $\lambda\to\infty$.
\end{enumerate}
\end{proposition}

\begin{proof}
The estimate on $\xExp_u\big[\hat{T}^\circlearrowleft_u\,\big|\,\hat{T}^\circlearrowleft_u<\hat{T}_v\big]$ in the time-homogeneous setting is implicit in the proof of~\cite[Theorem~1.2]{HolNarTaa16}. Let
\begin{itemize}
\item $\varepsilon\isdef\xPr_u(\hat{T}_v<\hat{T}^\circlearrowleft_u) = \xPr(B=1)$,
\item $\mu\isdef\xExp_u[\hat{T}^\circlearrowleft_u\,|\,\hat{T}^\circlearrowleft_u<\hat{T}_v]=\xExp[\delta T\,|\,B=0]$,
\item $\eta\isdef\xExp_u[\hat{T}_v\,|\,\hat{T}_v<\hat{T}^\circlearrowleft_u]=\xExp[\delta T\,|\,B=1]$,
\item $M\isdef\xExp_u[\hat{T}_v]$.
\end{itemize}
We know that
\begin{align}
\varepsilon &= \frac{1}{\pi(u)\effR{u}{v}} \asymp \frac{1}{\Gamma} = \smallo(1),
\qquad \text{as $\lambda\to\infty$} \;.
\end{align}
The first equality is~\cite[Eq.~(A.5)]{HolNarTaa16}, which is standard. The second equality follows from the fact that $\effR{u}{v}\asymp\Psi(u,v)$ (Proposition~A.2 of~\cite{HolNarTaa16}), and the last equality is by Lemma~\ref{lem:no-deep-well}.  Furthermore, we know that
\begin{align}
\label{Mprop}
M &= \pi(u)\effR{u}{v}[1+\smallo(1)] \;, \qquad\text{as $\lambda\to\infty$}.
\end{align}
Under Assumption~\ref{keyassump}\ref{keyassump:general}, \eqref{Mprop} is the same as~\cite[Eq.~(4.1)]{HolNarTaa16}. Note that in our setting, hypothesis~(H0) of~\cite{HolNarTaa16} follows from $\beta_U\abs{U}<\beta_V\abs{V}$. Under Assumption~\ref{keyassump}\ref{keyassump:complete-bipartite}, \eqref{Mprop} is the same as~\cite[Eq.~(2.5) in Example~2.1]{HolNarTaa16}.

Now, note that $M=(\nicefrac{1}{\varepsilon}-1)\mu+\eta$. Since $\mu\geq 1$, it follows that $\mu=1+\smallo(1)$ and $\eta=\smallo(M)$ as $\lambda\to\infty$.
\end{proof}

For Conditions~\ref{item:task:conditional-u:time-homogeneous} and~\ref{item:task:conditional-v:time-homogeneous}, we use an asymptotic result on conditional tail probabilities of exit times established in~\cite{HolNarTaa16}.

\begin{proposition}
\label{prop:task:conditional-u/v:time-homogeneous:proof}
Conditions~\ref{item:task:conditional-u:time-homogeneous} and~\ref{item:task:conditional-v:time-homogeneous} are satisfied provided $C\succ\check{\Gamma}\log\gamma$.
\end{proposition}

\begin{proof}
That~\ref{item:task:conditional-u:time-homogeneous} is satisfied follows from~\cite[Proposition~B.9]{HolNarTaa16} if we set $A\isdef\pspace{X}\setminus\{u,v\}$, $B_1\isdef\{u\}$ and $B_2\isdef\{v\}$. Namely, note that for this case $\kappa\asymp\nicefrac{1}{\gamma}$.  Set $\rho\isdef C/\check{\Gamma}$.  Then, by the above-mentioned proposition, there exists a constant $\alpha<1$ such that
\begin{equation}
\begin{aligned}
\sup_{x\notin\{u,v\}}\xPr^{(s)}_x\big(\hat{T}_u>C\,\big|\,\hat{T}_u<\delta\hat{T}_v\big) &\preceq
\alpha^\rho\kappa^{-\abs{\pspace{X}\setminus\{u,v\}}} \\
&= \alpha^{C/\check{\Gamma}}\gamma^{\abs{\pspace{X}}-2}
= \ee^{(C/\check{\Gamma})\log\alpha + (\abs{\pspace{X}}-2)\log\gamma},
\end{aligned}
\end{equation}
which tends to~$0$ as $\lambda\to\infty$.

The argument for~\ref{item:task:conditional-v:time-homogeneous} is similar.  First, condition on the state of the Markov chain after one step. The state of the chain after one step is $x\notin\{u,v\}$. Now, apply~\cite[Proposition~B.9]{HolNarTaa16} with $A\isdef\pspace{X}\setminus\{u,v\}$, $B_1\isdef\{v\}$ and $B_2\isdef\{u\}$ and note that $\kappa\asymp\nicefrac{1}{\gamma}$.
\end{proof}

\begin{proposition}
\label{prop:task:last-excursion:time-homogeneous:proof}
Subject to Assumption~{\rm \ref{keyassump}}, Condition~\ref{item:task:last-excursion:time-homogeneous} is satisfied.
\end{proposition}

\begin{proof}
Let $Q=Q(\lambda)$ be such that 
\begin{align}
\xExp^{(s)}_u\big[T_v\,\big|\,T_v<T^\circlearrowleft_u\big] 
&\prec Q \prec \xExp^{(s)}_u[T_v]\;, \qquad \lambda\to\infty.
\end{align}
By Proposition~\ref{prop:task:truncated-mean-and-conditional-return:time-homogeneous:proof}\ref{item:tasks:time-homogeneous:paper-1:conditional-crossover}, such a function $Q$ can be chosen. Let $\delta T\isdef T_v-S$, and note that $\delta T$ has distribution $\xPr_u\big(T_v\in\cdot\,\big|\,T_v<T^\circlearrowleft_u\big)$ and is independent of $S$. Hence
\begin{align}
\xPr_u(S\leq M\tau<T_v) 
&= \xPr_u(M\tau - \delta T\leq S < M\tau) \nonumber \\
&\leq \begin{multlined}[t]
\xPr(\delta T\geq Q)\xPr_u(M\tau - \delta T\leq S< M\tau \,|\, \delta T\geq Q) \\
+ \xPr(\delta T < Q)\xPr_u(M\tau - Q\leq S< M\tau).
\end{multlined}
\end{align}
By the Markov inequality, $\xPr(\delta T\geq Q)=\smallo(1)$. To bound the second term, recall from~\cite[Theorem~2.1]{HolNarTaa16} that the distribution of $S/\xExp[S]$ converges weakly as $\lambda\to\infty$ to an exponential distribution with rate~$1$. Hence
\begin{align}
\xPr_u(M\tau - Q\leq S< M\tau) 
&= \xPr_u\left(\frac{M\tau}{\xExp[S]} - \frac{Q}{\xExp[S]} 
\leq \frac{S}{\xExp[S]} < \frac{M\tau}{\xExp[S]}\right). 
\end{align}
Since $Q\prec\xExp[S]$ and the exponential distribution has no atom, we find that $\xPr_u(M\tau - Q\leq S< M\tau) = \smallo(1)$.
\end{proof}

In summary, we have the following solution for the parameters ensuring that the time-homogeneous Conditions~\ref{item:task:truncated-mean:time-homogeneous}--\ref{item:task:last-excursion:time-homogeneous} of the short-time regularity conditions are satisfied.

\begin{proposition}[\textbf{Short-time regularity: time-homogeneous setting}]
\label{prop:conditions-satisfied:time-homogeneous}
Subject to Assumption~\ref{keyassump} and the constraints 
\begin{align}
\label{Msand}
\frac{\log\gamma}{\gamma}\check{\Gamma}\prec M
\prec\frac{\log\gamma}{\gamma}\check{\Gamma}\Gamma,
\qquad \lambda\to\infty,
\end{align}
there is a choice of the parameters $C$, $m$, $\delta s$, $k$ for which Conditions~\ref{item:task:parameter:C}--\ref{item:task:parameter:E-vs-m} and~\ref{item:task:truncated-mean:time-homogeneous}--\ref{item:task:last-excursion:time-homogeneous} are satisfied. In particular, a proper choice is
\begin{align}
\label{eq:parameters:explicit-choices:time-homogeneous}
C & \isdef k \isdef A_1\check{\Gamma}\log\gamma, &
\delta s &\isdef A_1A_2\frac{\log\gamma}{\gamma}\check{\Gamma}, &
m & \isdef \frac{A_3 M\gamma}{A_1A_2\check{\Gamma}\log\gamma},
\end{align}
for any choices of $A_1=A_1(\lambda)$, $A_2=A_2(\lambda)$ and $A_3=A_3(\lambda)$ that tend to infinity sufficiently slowly as $\lambda\to\infty$.
\end{proposition}

\begin{proof}
Conditions~\ref{item:task:truncated-mean:time-homogeneous}, \ref{item:task:conditional-return:time-homogeneous} and~\ref{item:task:last-excursion:time-homogeneous} are guaranteed by Proposition~\ref{prop:task:truncated-mean-and-conditional-return:time-homogeneous:proof} and Proposition~\ref{prop:task:last-excursion:time-homogeneous:proof}. According to Proposition~\ref{prop:task:conditional-u/v:time-homogeneous:proof}, Conditions~\ref{item:task:conditional-u:time-homogeneous} and~\ref{item:task:conditional-v:time-homogeneous} are satisfied if $C\succ\check{\Gamma}\log\gamma$ and $k\succ\check{\Gamma}\log\gamma$. So, set $C\isdef k\isdef A_1\check{\Gamma}\log\gamma$ for some $A_1\succ 1$ to be chosen later.
	
By Lemma~\ref{lem:tasks:parameters:simplified}, Conditions~\ref{item:task:parameter:C}--\ref{item:task:parameter:E-vs-m} are guaranteed if
\begin{align}
\label{conds}
C &= \Gamma\smallo(1),
&m &= \Gamma\smallo(1),
&\delta s &= \smallo(M)\;,
&k &\leq \tfrac{1}{2}\gamma \delta s,
&m\,\delta s &\succ M.
\end{align}
The first condition in \eqref{conds} is satisfied if $A_1\check{\Gamma}\log\gamma\prec\Gamma$. By Lemma~\ref{lem:no-deep-well}, the latter condition is satisfied as long as $A_1\to\infty$ sufficiently slowly as $\lambda\to\infty$.
	
In order to satisfy the third and the fourth condition in \eqref{conds}, we must be able to choose $\delta s$ such that $2A_1 \frac{\log\gamma}{\gamma}\check{\Gamma}\leq\delta s\prec M$. Set $\delta s \isdef A_1A_2\frac{\log\gamma}{\gamma}\check{\Gamma}$, for some $A_2\succ 1$ to be chosen later such that $A_1A_2\frac{\log\gamma}{\gamma}\check{\Gamma}\prec M$.
	
In order to satisfy the second and the fifth condition in \eqref{conds}, we must be able to choose $m$ such that $m\prec\Gamma$ and $mA_1A_2\frac{\log\gamma}{\gamma}\check{\Gamma}\succ M$. This can be done if $M\prec A_1A_2\frac{\log\gamma}{\gamma}\check{\Gamma}\Gamma$, in which case we can set $m\isdef A_3M/A_1A_2\frac{\log\gamma}{\gamma}\check{\Gamma}$ with $A_3$ tending to infinity sufficiently slowly as $\lambda\to\infty$.
	
Finally, if $\frac{\log\gamma}{\gamma}\check{\Gamma}\prec M\prec\frac{\log\gamma}{\gamma}\check{\Gamma}\Gamma$, then we are able to choose $A_1,A_2\succ 1$ such that all five conditions in \eqref{conds} are satisfied.
\end{proof}

In conclusion, in the time-homogeneous setting we have proved that Conditions~\ref{item:task:parameter:C}--\ref{item:task:parameter:E-vs-m} and~\ref{item:task:truncated-mean:time-homogeneous}--\ref{item:task:last-excursion:time-homogeneous} can be simultaneously met subject to mild conditions on the time scale $M$. These conditions capture the \emph{critical regime}. Indeed, since $\nu = \varepsilon\gamma \asymp \gamma/\Gamma$, the critical regime corresponds to $M \asymp \Gamma/\gamma$. Since $\check{\Gamma}\log\gamma \succ 1$ as $\lambda\to\infty$, the upper bound in \eqref{Msand} is matched. Since $\Gamma\succ\check{\Gamma}\log\gamma$ and $(\log \gamma)/\gamma\to 0$ as $\gamma\to\infty$, also the lower bound is matched.


\subsection{Short-time regularity in the time-inhomogeneous setting}
\label{sec:reginhom}

In this section we establish the short-time regularity Conditions~\ref{item:task:truncated-mean}--\ref{item:task:last-excursion}.  We focus on the critical regime $M\nu(0)\asymp 1$ when $M\prec\lambda$, or $M\asymp\lambda$ and $M\tau<\frac{c_U}{\mu_U}\lambda$. Inspired by Proposition~\ref{prop:conditions-satisfied:time-homogeneous}, we choose
\begin{align}
\label{eq:parameters:explicit-choices:time-inhomogeneous}
C & \isdef k \isdef A_1\check{\Gamma}^{(0)}\log\gamma(0), &
\delta s &\isdef A_1A_2\frac{\log\gamma(0)}{\gamma(0)}\check{\Gamma}^{(0)}, &
m & \isdef \frac{A_3 M\gamma(0)}{A_1A_2\check{\Gamma}^{(0)}\log\gamma(0)},
\end{align}
where $A_1=A_1(\lambda)$, $A_2=A_2(\lambda)$ and $A_3=A_3(\lambda)$ tend to infinity as $\lambda\to\infty$ sufficiently slowly. By Proposition~\ref{prop:conditions-satisfied:time-homogeneous} and Proposition~\ref{prop:critical-regime:orders-of-magnitude}, we already know that these choices satisfy~\ref{item:task:parameter:C}--\ref{item:task:parameter:E-vs-m} provided that
\begin{align}
\label{Msand:inhomogeneous}
\frac{\log\gamma(0)}{\gamma(0)}\check{\Gamma}^{(0)}\prec M
\prec\frac{\log\gamma(0)}{\gamma(0)}\check{\Gamma}^{(0)}\Gamma^{(0)},
\qquad \lambda\to\infty.
\end{align}
In the critical regime $M\gamma(0)\check{\varepsilon}(0)\asymp 1$, the latter condition is satisfied since $M\gamma(0)\asymp\Gamma^{(0)}$ and $\Gamma^{(0)}\succ\check{\Gamma}^{(0)}\log\gamma(0)$.

\paragraph*{Comparison}

Conditions~\ref{item:task:truncated-mean} and~\ref{item:task:conditional-return} will be handled by direct comparison with the time-homogeneous process analysed previously.

\begin{lemma}[\textbf{Short-term comparison of probabilities of events}]
\label{lem:markov-chain:short-term-comparison}
Let $(X_i)_{i\in\NN}$ and $(X'_i)_{i\in\NN}$ be two discrete-time Markov chains with the same finite state space $\pspace{X}$ and (possibly time-inhomogeneous) transition probabilities $K_i(x,y)$ and $K'_i(x,y)$ for $x,y\in\pspace{X}$. Suppose that $\beta>0$ is such that $K_i(x,y)\leq(1+\beta)K'_i(x,y)$ for all $i\in\NN$ and $x,y\in\pspace{X}$. Then, for every state $u\in\pspace{X}$ and every event $E\in\pspace{X}^n$ (with $n\in\NN$),
\begin{align}
\xPr\big( (X_1,\ldots,X_n)\in E \,\big|\, X_0=u \big) &\leq
(1+\beta)^n \xPr\big( (X'_1,\ldots,X'_n)\in E \,\big|\, X'_0=u \big).
\end{align}
\end{lemma}

\begin{lemma}[\textbf{Comparison of transition probabilities at different times}]
\label{lem:transition-probabilities:comparison}
For the time-inhomogeneous process under consideration, the transition matrices $K^{(s)}$ and $K^{(s')}$ at times $0\leq s,s'<\frac{c_U}{\mu_U}\lambda$ satisfy $K^{(s)}(x,y)\leq(1+\beta_{s,s'})K^{(s')}(x,y)$ for every $x,y\in\pspace{X}$, where $\beta_{s,s'}\asymp\abs{s-s'}/\lambda$ as $\lambda\to\infty$. The time instants $s,s'$ are allowed to depend on $\lambda$, but must satisfy $\abs{s-s'}=\smallo(\lambda)$ and $c_U\lambda-\mu_Us, c_U\lambda-\mu_Us'\succeq\lambda$ as $\lambda\to\infty$.
\end{lemma}

\begin{proof}
Recall that $K^{(s)}(x,y)$ is either $0$ or is of the form $\lambda_U(s)/\gamma(s)$, $\lambda_V(s)/\gamma(s)$ or $1/\gamma(s)$. We show that, for every $0\leq s,s'\leq\frac{c_U}{\mu_U}\lambda$,
\begin{align}
\frac{1}{1+\theta_{s,s'}}&\leq\frac{\lambda_U(s)}{\lambda_U(s')}\leq 1+\theta_{s,s'}, 
&\frac{1}{1+\theta_{s,s'}}&\leq\frac{\lambda_V(s)}{\lambda_V(s')}\leq 1+\theta_{s,s'},
\end{align}
where $\theta_{s,s'}\asymp\abs{s-s'}/\lambda$. Indeed, 
\begin{align}
\frac{\lambda_U(s)}{\lambda_U(s')} 
&= \frac{(c_U\lambda-\mu_U s)^{\beta_U}}{(c_U\lambda-\mu_U s')^{\beta_U}} 
=  \left[1 + \frac{\mu_U(s'-s)}{c_U\lambda-\mu_Us'}\right]^{\beta_U},
\end{align}
which is $\leq 1+\theta_{s,s'}$ when $\theta_{s,s'}\geq\frac{2\beta_U\mu_U\abs{s'-s}}{c_U\lambda-\mu_Us'} \asymp \abs{s-s'}/\lambda$.  The opposite inequality follows by symmetry. The inequality for $\lambda_V$ follows similarly.

Since $\gamma(s)$ has the same order of magnitude as $\lambda_V(s)$, it follows that
\begin{align}
\frac{1}{1+\theta_{s,s'}}&\leq\frac{\gamma(s)}{\gamma(s')}\leq 1+\theta_{s,s'}
\end{align}
and hence
\begin{align}
\frac{1}{(1+\theta_{s,s'})^2}&\leq\frac{K^{(s)}(x,y)}{K^{(s)}(x,y)}\leq (1+\theta_{s,s'})^2
\end{align}
for every $x,y\in\pspace{X}$. Thus, the claim holds with $\beta_{s,s'}=2\theta_{s,s'}+\theta_{s,s'}^2\asymp\theta_{s,s'}\asymp\abs{s-s'}/\lambda$.
\end{proof}

The combination of Lemma~\ref{lem:markov-chain:short-term-comparison} and Lemma~\ref{lem:transition-probabilities:comparison} leads us to the following proposition.

\begin{proposition}[\textbf{Short-term comparison with time-homogeneous process}]
\label{prop:time-homogeneous-inhomogeneous:comparison:short-term}
Let $N=N(\lambda)$ be a non-negative integer such that $N\succ 1$ as $\lambda\to\infty$. Consider the time scaling $t=M\tau$ with $M=M(\lambda)$ and $\tau\in[0,\infty)$, and suppose that either $1\preceq M\prec\lambda$, or $M\asymp\lambda$ and $\tau<\frac{c_U}{\mu_U}\frac{\lambda}{M}$. Set $T_0\isdef s$, and let $T_1,T_2,\ldots$ be the consecutive points of the Poisson process $\rv{\xi}$ after time~$s$. Then, for every state $x\in\pspace{X}$ and every event $E\in\pspace{X}^N$, there exists a constant $D_N$ such that
\begin{align}
\MoveEqLeft\nonumber
\xPr\big(\big[X(T_0), X(T_1),\ldots,X(T_N)\big]\in E\,\big|\,X(s)=x,\rv{\xi}=\xi\big) \\
&\leq
D_N\xPr^{(s)}\big(\big[\hat{X}(0), \hat{X}(1),\ldots,\hat{X}(N)\big]\in E\,\big|\,\hat{X}(0)=x\big)
\end{align}
as $\lambda\to\infty$, uniformly in $s\in[0,M\tau]$ and $\xi$ satisfying $\xi([s,s+\lambda/N])\geq N$.
\end{proposition}

\begin{proof}
Let $\Delta\isdef\lambda/N$ and note that $\Delta=\smallo(\lambda)$. By Lemma~\ref{lem:transition-probabilities:comparison}, there exists a constant $B$ such that $K^{(s')}(x,y)\leq(1+B/N)K^{(s)}(x,y)$ for every $s'$ satisfying $s\leq s'\leq s+\Delta$ and every $x,y\in\pspace{X}$. Therefore, by Lemma~\ref{lem:markov-chain:short-term-comparison}, on the event $\big\{\rv{\xi}([s,s+\Delta])\geq N\big\}$,
\begin{align}
\MoveEqLeft\nonumber
\xPr\big(\big[X(T_0), X(T_1),\ldots,X(T_N)\big]\in E\,\big|\,\rv{\xi},X(s)=x\big) \\
&\leq
D_N\xPr^{(s)}\big(\big[\hat{X}(0), \hat{X}(1),\ldots,\hat{X}(N)\big]\in E\,\big|\,\hat{X}(0)=x\big)
\end{align}
where $D_N\isdef(1+B/N)^N\leq\ee^B$.
\end{proof}

\begin{lemma}
\label{lem:time-inhomogeneous:poisson:good-realizations}
Consider the time scaling $t=M\tau$ with $M=M(\lambda)$ and $\tau\in[0,\infty)$, and suppose that either $1\preceq M\prec\lambda$, or $M\asymp\lambda$ and $\tau<\frac{c_U}{\mu_U}\frac{\lambda}{M}$. Let $N=N(\lambda)$ be a non-negative integer such that $N^2\prec\gamma(0)\lambda$. Then, with probability $1-\smallo(1)$ as $\lambda\to\infty$, $\rv{\xi}([s,s+\lambda/N])\geq N$ for every $s\in[0,M\tau]$.
\end{lemma}

\begin{proof}
Let $\Delta\isdef\lambda/N$. Note that $X\isdef\rv{\xi}([s,s+\Delta/2])$ is a Poisson random variable with parameter $\int_s^{s+\Delta/2}\gamma(u)\dd u\asymp\gamma(0)\Delta/2$, where the last asymptotic equality holds by Proposition~\ref{prop:critical-regime:orders-of-magnitude}. By the Chebyshev inequality, uniformly in $s\in[0,M\tau]$,
\begin{align}
\xPr\big(\rv{\xi}([s,s+\Delta/2])<N\big) 
&= \xPr\big(\xExp[X]-X\geq\xExp[X]-N\big) \leq \frac{\xVar[X]}{(\xExp[X]-N)^2} \\
&\asymp \nonumber \frac{\gamma(0)\Delta/2}{(\gamma(0)\Delta/2-N)^2}
= \frac{2\gamma(0)\Delta}{\big(\gamma(0)\Delta\big)^2\big(1-\frac{4N^2}{\gamma(0)\lambda}\big)^2}
\asymp \frac{1}{\gamma(0)\Delta}.
\end{align}
Partition the interval $[0,M\tau]$ into segments of length $\Delta/2$. Note that if $\rv{\xi}$ has at least $N$ points in each of these segments, then $\rv{\xi}([s,s+\Delta])\geq N$ for every $s\in[0,M\tau]$. Hence
\begin{align}
\xPr\big(\text{$\rv{\xi}([s,s+\Delta])<N$ for some $s\in[0,M\tau]$}\big) &\preceq
\left\lceil\frac{M\tau}{\Delta/2}\right\rceil\times\frac{1}{\gamma(0)\Delta}
\preceq \frac{\lambda N^2}{\gamma(0)\lambda^2} = \smallo(1),
\end{align}
as claimed.
\end{proof}

\paragraph*{Verification of \ref{item:task:truncated-mean}--\ref{item:task:last-excursion}}

In what follows we assume the following:
\begin{itemize}
\item[$(\ast)$]
Suppose that Assumptions~\textup{\ref{keyassump}--\ref{keyassumpextra}} are satisfied. Consider the time scaling $s=M\sigma$ with $M=M(\lambda)$ and $\sigma\in[0,\infty)$, and suppose that either $1\preceq M\prec\lambda$, or $M\asymp\lambda$ and $\sigma<\frac{c_U}{\mu_U}\frac{\lambda}{M}$. Suppose further that $M\check{\nu}(0)\asymp 1$.
\end{itemize}

Applying Proposition~\ref{prop:time-homogeneous-inhomogeneous:comparison:short-term} and Lemma~\ref{lem:time-inhomogeneous:poisson:good-realizations}, we can now establish the short-time regularity Conditions~\ref{item:task:truncated-mean} and~\ref{item:task:conditional-return} by comparison with the time-homogeneous setting.

\begin{proposition}
\label{prop:time-inhomogeneous:short-term-regularity:I-II}
Let $C$ be as in~\eqref{eq:parameters:explicit-choices:time-inhomogeneous}.
Subject to $(\ast)$, Conditions~\ref{item:task:truncated-mean} and~\ref{item:task:conditional-return} are satisfied. 
\end{proposition}

\begin{proof}
First recall, from monotonicity and assumption~\eqref{eq:assumption:freezing-at-the-end}, that almost surely
\begin{align}
\xPr\big(\delta\hat{T}^\circlearrowleft_u>\delta\hat{T}_v\,\big|\,\rv{\xi},X(s)=u\big) &\leq
\check{\varepsilon}(M\tau)\asymp\check{\varepsilon}(0) = \smallo(1).
\end{align}
Therefore it is enough to show that
\begin{equation}
\begin{aligned}
\xExp\big[\delta\hat{T}^\circlearrowleft_u(s)\indicator{\delta\hat{T}^\circlearrowleft_u(s)\leq C+1}\,\big|\,\rv{\xi}=\xi,\,
X(s)=u\big] &= 1+\smallo(1) \;, \\
C\xPr\big(\delta\hat{T}^\circlearrowleft_u(s)>C+1\,\big|\,\rv{\xi}=\xi,\,X(s)=u\big) &= \smallo(1),
\end{aligned}
\end{equation}
uniformly in $s\in[0,M\tau]$ and in $\xi$ belonging to a set $\Xi_{M\tau}$ satisfying $\xPr(\rv{\xi}\in\Xi_{M\tau})=1-\smallo(1)$.
Note that both these statements concern events that depend on no more than $N\isdef C+1$ ticks of the Poisson clock starting from $s$. Thus, applying Proposition~\ref{prop:time-homogeneous-inhomogeneous:comparison:short-term}, we get
\begin{equation}
\label{eq:time-inhomogeneous:short-term-regularity:I-II:proof}
\begin{aligned}
\xExp\big[\delta\hat{T}^\circlearrowleft_u(s)\indicator{\delta\hat{T}^\circlearrowleft_u(s)\leq C+1}\,\big|\,\rv{\xi}=\xi,\,
X(s)=u\big] &\leq D_N \times \xExp^{(s)}_u\big[\hat{T}^\circlearrowleft_u\indicator{\hat{T}^\circlearrowleft_u\leq C+1}\big] \;, \\
C\xPr\big(\delta\hat{T}^\circlearrowleft_u(s)>C+1\,\big|\,\rv{\xi}=\xi,\,X(s)=u\big) &\leq
D_N \times C\xPr^{(s)}_u\big(\hat{T}^\circlearrowleft_u>C+1\big),
\end{aligned}
\end{equation}
uniformly in $s\in[0,M\tau]$ and in $\xi$ satisfying $\xi([s,s+\lambda/N])\geq N$.

Next, note that
\begin{align}
\label{eq:time-inhomogeneous:short-term-regularity:I-II:proof:N^2}
N^2 &= (C+1)^2 \asymp A_1^2(\check{\Gamma}^{(0)})^2(\log\gamma(0))^2 \;.
\end{align}
Let us first argue that, if $A_1$ is chosen to grow sufficiently slowly, then the right-hand side of~\eqref{eq:time-inhomogeneous:short-term-regularity:I-II:proof:N^2} is $\prec\Gamma^{(0)}$, which in turn is $\preceq\gamma(0)\lambda$.  To see the first inequality, note that $\Gamma^{(0)}$ and $\check{\Gamma}^{(0)}$ are rational function of $\lambda_U\asymp\lambda^{\beta_U}$ and $\lambda_V\asymp\lambda^{\beta_V}$ (see the proof of Lemma~\ref{lem:no-deep-well}).  Since $(\check{\Gamma}^{(0)})^2$ is assumed to be of smaller order than $\Gamma^{(0)}$, it follows that $(\check{\Gamma}^{(0)})^2(\log\gamma(0))^2\asymp(\check{\Gamma}^{(0)})^2(\log\lambda)^2$ is of smaller order than $\Gamma^{(0)}$ as well.  To see the second inequality, note that $\Gamma^{(0)}\asymp 1/\check{\varepsilon}(0)=\gamma(0)/\nu(0)\asymp \gamma(0)M\preceq\gamma(0)\lambda$.
Thus, applying Lemma~\ref{lem:time-inhomogeneous:poisson:good-realizations}, we find that the equalities in~\eqref{eq:time-inhomogeneous:short-term-regularity:I-II:proof} hold uniformly in $s\in[0,M\tau]$ and in $\xi$ in a set $\Xi_{M\tau}$ satisfying $\xPr(\rv{\xi}\in\Xi_{M\tau})=1-\smallo(1)$ as $\lambda\to\infty$.

To prove the claim, it remains to show that
\begin{align}
\xExp^{(s)}_u\big[\hat{T}^\circlearrowleft_u\indicator{\hat{T}^\circlearrowleft_u\leq C+1}\big] &= 1+\smallo(1)\;, &
C\xPr^{(s)}_u\big(\hat{T}^\circlearrowleft_u>C+1\big) &= \smallo(1),
\end{align}
uniformly in $s\in[0,M\tau]$. But these follow from Proposition~\ref{prop:task:truncated-mean-and-conditional-return:time-homogeneous:proof}\ref{item:tasks:time-homogeneous:paper-1:conditional-return} and the fact that $\xPr^{(s)}_u\big(\hat{T}^\circlearrowleft<\hat{T}_v\big)=1-\check{\varepsilon}(s)=1-\smallo(1)$ uniformly in $s\in[0,M\tau]$.
%
\end{proof}

A similar approach establishes~\ref{item:task:conditional-u} and~\ref{item:task:conditional-v}.  In this case, we cannot directly apply Proposition~\ref{prop:task:conditional-u/v:time-homogeneous:proof}.  Instead, we need to redo the proof of Proposition~B.9 of~\cite{HolNarTaa16}.

\begin{proposition}
\label{prop:time-inhomogeneous:short-term-regularity:III}
Let $C$ be as in~\eqref{eq:parameters:explicit-choices:time-inhomogeneous}.
Subject to Condition~$(\ast)$, Condition~\ref{item:task:conditional-u} is satisfied. 
\end{proposition}

\begin{proof}
Following the proof of~\cite[Proposition~B.9]{HolNarTaa16}, we can write
\begin{align}
\MoveEqLeft[3]\nonumber
\circled{\sun} \isdef \sup_{x\notin\{u,v\}}\xPr\big(\delta\hat{T}_u(s)>C\,\big|\,\rv{\xi}=\xi,\,X(s)=x,\,
\delta\hat{T}_u(s)<\delta\hat{T}_v(s)\big) \\
&= \label{eq:time-inhomogeneous:short-term-regularity:III:proof:a}
\sup_{x\notin\{u,v\}}\frac{\xPr\big(\delta\hat{T}_v(s)>\delta\hat{T}_u(s)>C\,\big|\,\rv{\xi}=\xi, X(s)=x\big)}{\xPr\big(\delta\hat{T}_v(s)>\delta\hat{T}_u(s)\,\big|\,\rv{\xi}=\xi, X(s)=x\big)} \\
&\leq \sup_{x\notin\{u,v\}}\frac{\xPr\big(\delta\hat{T}_{\{u,v\}}(s)>C\,\big|\,\rv{\xi}=\xi, X(s)=x\big)}{\xPr\big(\delta\hat{T}_v(s)>\hat{T}_u(s)\,\big|\,\rv{\xi}=\xi, X(s)=x\big)}. \nonumber
\end{align}
Applying Proposition~\ref{prop:time-homogeneous-inhomogeneous:comparison:short-term}, we see that the numerator in the right-hand side of~\eqref{eq:time-inhomogeneous:short-term-regularity:III:proof:a} is bounded from above by $D_C\xPr^{(s)}_x(\hat{T}_{\{u,v\}}>C)$ uniformly in $\xi$ and $s\in[0,M\tau]$ satisfying $\xi([s,s+\lambda/C])\geq C$.  Lemma~\ref{lem:time-inhomogeneous:poisson:good-realizations} together with the assumption $\check{\Gamma}^{(0)}\prec\sqrt{\Gamma^{(0)}}$ implies that for $\xi$ in a set $\Xi_{M\tau}$ satisfying $\xPr(\rv{\xi}\in\Xi_{M\tau})=1-\smallo(1)$, the inequality $\xi([s,s+\lambda/C])\geq C$ holds for every $s\in[0,M\tau]$ (see the proof of Proposition~\ref{prop:time-inhomogeneous:short-term-regularity:I-II}).  The denominator in the right-hand side of~\eqref{eq:time-inhomogeneous:short-term-regularity:III:proof:a} is by monotonicity bounded from below by $\xPr^{(s)}_x(\hat{T}_v>\hat{T}_u)$. Hence
\begin{align}
\circled{\sun} &\leq \sup_{x\notin\{u,v\}} \frac{D_C\xPr^{(s)}_x(\hat{T}_{\{u,v\}}>C)}{\xPr^{(s)}_x(\hat{T}_v>\hat{T}_u)}
\end{align}
uniformly in $\xi\in\Xi_{M\tau}$ and $s\in[0,M\tau]$.  By~\cite[Proposition~B.8]{HolNarTaa16}, $\xPr^{(s)}_x(\hat{T}_{\{u,v\}}>C)$ is bounded from above by $\alpha^\rho$ where $\alpha<1$ is a constant and $\rho\isdef C/\check{\Gamma}^{(s)}=A_1\log\gamma(s)$.  Let $\kappa(s)\isdef\min\{K^{(s)}(a,b): a,b\in\pspace{X}, K(a,b)>0\} = 1/\gamma(s)$.  Let $w$ be a simple path from $x$ to $u$ that does not pass through $v$.  The length of $w$ is no larger than $\abs{\pspace{X}}-2$.  Therefore $\xPr^{(s)}_x(\hat{T}_v>\hat{T}_u)\geq\xPr^{(s)}_x(\text{$\hat{X}$ follows $w$})\geq\kappa(s)^{\abs{\pspace{X}}-2}$. Hence
\begin{align}
\circled{\sun} &\leq D_C \alpha^{A_1\log\gamma(s)}\gamma(s)^{\abs{\pspace{X}}-2} =
D_C\, \ee^{A_1\log\gamma(s)\log\alpha + (\abs{\pspace{X}}-2)\log\gamma(s)}
\end{align}
uniformly in $\xi\in\Xi_{M\tau}$ and $s\in[0,M\tau]$.  Now recall from Proposition~\ref{prop:critical-regime:orders-of-magnitude} that $\gamma(0)\asymp\gamma(s)$ for every $0\leq s\leq M\tau$.  Since $A_1\succ 1$, it follows that $\circled{\sun}=\smallo(1)$ uniformly in $\xi\in\Xi_{M\tau}$ and $s\in\bar{\xi}\cap[0,M\tau]$.
\end{proof}

\begin{proposition}
\label{prop:time-inhomogeneous:short-term-regularity:IV}
Let $k$ and $\delta s$ be as in~\eqref{eq:parameters:explicit-choices:time-inhomogeneous}.
Subject to $(\ast)$, Condition~\ref{item:task:conditional-v} is satisfied. 
\end{proposition}

\begin{proof}
As in the proof of Proposition~\ref{prop:time-inhomogeneous:short-term-regularity:III} above, we start with observing that
\begin{align}
\MoveEqLeft[3]\nonumber
\circled{\leftmoon} \isdef \xPr\big(\delta\hat{T}_v(s)>k\,\big|\,\rv{\xi}=\xi,\,X(s)=u,\, \delta\hat{T}_v(s)<\delta\hat{T}^\circlearrowleft_u(s)\big) \\
&= \label{eq:time-inhomogeneous:short-term-regularity:IV:proof:a}
\frac{\xPr\big(\delta\hat{T}^\circlearrowleft_u(s)>\delta\hat{T}_v(s)>k\,\big|\,\rv{\xi}=\xi, X(s)=u\big)}{
\xPr\big(\delta\hat{T}^\circlearrowleft_u(s)>\delta\hat{T}_v(s)\,\big|\,\rv{\xi}=\xi, X(s)=u\big)} \\
&\leq \frac{\xPr\big(\delta\hat{T}_{\{u,v\}}(s)>k\,\big|\,\rv{\xi}=\xi, X(s)=u\big)}{
\xPr\big(\delta\hat{T}^\circlearrowleft_u(s)>\delta\hat{T}_v(s)\,\big|\,\rv{\xi}=\xi, X(s)=u\big)}. \nonumber
\end{align}
By monotonicity and assumption~\eqref{eq:assumption:freezing-at-the-end}, the denominator can be bounded from below by $\xPr^{(M\tau)}_u(\hat{T}^\circlearrowleft_u>\hat{T}_v)$.  Applying Lemma~\ref{lem:markov-chain:short-term-comparison} and Lemma~\ref{lem:transition-probabilities:comparison}, we see that, for $s$ satisfying $\xi((s,s+\delta s))\geq k$, the numerator of the right-hand side of~\eqref{eq:time-inhomogeneous:short-term-regularity:IV:proof:a} is bounded from above by $\big(1+\frac{\delta s}{\lambda}\big)^k\xPr^{(s)}_u(\hat{T}^\circlearrowleft_{\{u,v\}}>k)$.  By the choice of $k$ and $\delta s$, and the assumption $\check{\Gamma}^{(0)}\prec\sqrt{\Gamma^{(0)}}$, we have
\begin{align}
\frac{\delta s}{\lambda}k &= \frac{A_2 A_1^2 (\check{\Gamma}^{(0)})^2(\log\gamma(0))^2}{\gamma(0)\lambda}
= \bigo(1)
\end{align}
as long as $A_1$ and $A_2$ grow sufficiently slowly (see the proof of Proposition~\ref{prop:time-inhomogeneous:short-term-regularity:I-II}).  Therefore,
\begin{align}
\label{eq:time-inhomogeneous:short-term-regularity:IV:proof:b}
\circled{\leftmoon} &\leq \bigo(1)\frac{\xPr^{(s)}_u(\hat{T}^\circlearrowleft_{\{u,v\}}>k)}{\xPr^{(s)}_u(\hat{T}^\circlearrowleft_u>\hat{T}_v)}
\end{align}
uniformly in $\xi$ and $s\in[0,M\tau]$ satisfying $\xi((s,s+\delta s))\geq k$, where we have used the fact that $\xPr^{(M\tau)}_u(\hat{T}^\circlearrowleft_u>\hat{T}_v)\asymp\xPr^{(s)}_u(\hat{T}^\circlearrowleft_u>\hat{T}_v)$ by Proposition~\ref{prop:critical-regime:orders-of-magnitude}.  To bound the right-hand side of~\eqref{eq:time-inhomogeneous:short-term-regularity:IV:proof:b}, we first condition on the first step of the Markov chain in the numerator and then proceed as in the proof of Proposition~\ref{prop:time-inhomogeneous:short-term-regularity:III}.  We find that $\circled{\leftmoon}=\smallo(1)$ uniformly in $\xi$ and $s\in[0,M\tau]$ satisfying $\xi((s,s+\delta s))\geq k$.
\end{proof}

\begin{proposition}
\label{prop:time-inhomogeneous:short-term-regularity:V}
Subject to $(\ast)$, Condition~\ref{item:task:last-excursion} is satisfied.
\end{proposition}

\begin{proof}
We loosely follow the argument in the proof of Proposition~\ref{prop:task:last-excursion:time-homogeneous:proof}.  Since a priori we do not know if $S/\xExp[S]$ (or $T_v/\xExp_u[T_v]$) converges in distribution to a continuous random variable, we use a more abstract argument.

Let $(\lambda_n)_{n\in\NN}$ be a sequence going to infinity.  By Helly's selection theorem (see e.g.,~\cite[Theorem~3.2.6]{Dur10}), there exists a subsequence $(\lambda_{n(i)})_{i\in\NN}$ and a right-continuous, non-increasing function $H:[0,\infty)\to[0,1]$ such that
\begin{align}
\lim_{\substack{\lambda\isdef\lambda_{n(i)}\\ i\to\infty}}\xPr_u\big(T_v/M>\tau\big) &= H(\tau)
\end{align}
for every $\tau\in[0,\infty)$ that is a continuity point of $H$.  As $H$ is non-increasing, it has at most countably many discontinuity points.  It is therefore enough to show that 
\begin{align}
\label{eq:time-inhomogeneous:short-term-regularity:V:limit:main}
\lim_{\substack{\lambda\isdef\lambda_{n(i)}\\ i\to\infty}}\xPr_u(S\leq M\tau<T_v) &= 0
\end{align}
for every continuity point $\tau\in[0,\infty)$ of $H$.

Let $\delta T\isdef T_v - S$.  Fix $\varepsilon>0$ and choose $Q=Q(\lambda)$ such that
\begin{align}
\label{eq:time-inhomogeneous:short-term-regularity:V:Q:requirements}
\xPr_u(\delta T>Q)\leq\varepsilon + \smallo(1)	\qquad\text{and}\qquad	Q\prec M \;.
\end{align}
Namely, let $k$ and $\delta s$ be as in~\eqref{eq:parameters:explicit-choices:time-inhomogeneous}.  We claim that $Q\isdef\delta s$ satisfies~\eqref{eq:time-inhomogeneous:short-term-regularity:V:Q:requirements}.  Indeed, by the discussion after~\eqref{eq:parameters:explicit-choices:time-inhomogeneous}, $\delta s\prec M$.  By monotonicity and~\eqref{eq:time-homogeneous:success-prob-vs-expectation}, $\xExp_u[S]\leq\xExp_u[T_v]\preceq\xExp^{(0)}_u[T_v]\asymp 1/\check{\nu}(0)\asymp M$.  Let $\tau_0\in[0,\infty)$ be large enough such that $\xPr_u(S>M\tau_0)<\varepsilon$.  Then,
\begin{align}
\label{eq:time-inhomogeneous:short-term-regularity:V:Q:choice}
\xPr_u(T_v - S>\delta s) &\leq
\begin{multlined}[t]
\xPr(S>M\tau_0) + \xPr\big(\rv{\xi}\big((S,S+\delta s)\big)<k\,\big|\,S\leq M\tau_0\big) \\
+ \xPr_u\big(T_v-S>\delta s\,\big|\,S\leq M\tau_0, \rv{\xi}\big((S,S+\delta s)\big)\geq k\big) \;.
\end{multlined}
\end{align}
The first term on the right-hand side is bounded by~$\varepsilon$.  The second term is $\smallo(1)$ because, given $S=s$ with $s\leq M\tau_0$, $\rv{\xi}\big((S,S+\delta s)\big)$ is a Poisson random variable with mean $\int_s^{s+\delta s}\gamma(x)\dd x$, which is $\asymp\gamma(0)\delta s\succ k$ uniformly in $s\in[0,M\tau_0]$.  To estimate the third term on the right-hand side of~\eqref{eq:time-inhomogeneous:short-term-regularity:V:Q:choice}, note that
\begin{align}
\MoveEqLeft\nonumber
\xPr_u\big(T_v-S>\delta s\,\big|\,S=s, \rv{\xi}\big((s,s+\delta s)\big)\geq k\big) \\
&\leq \xPr\big(\delta\hat{T}_v(s)>k\,\big|\,
\rv{\xi}\big((s,s+\delta s)\big)\geq k,
,X(s)=u,\;\delta\hat{T}_v(s)<\delta\hat{T}^\circlearrowleft_u(s)\big) \;.
\end{align}
According to~\ref{item:task:conditional-v}, the latter is $\smallo(1)$ uniformly in $s\in[0,M\tau_0]$.

Now, we have
\begin{align}
\xPr_u(S\leq M\tau<T_v) 
&= \xPr_u(M\tau< T_v \leq M\tau + \delta T) \nonumber \\
&\leq \begin{multlined}[t]
\xPr(\delta T> Q)\xPr_u(M\tau< T_v \leq M\tau + \delta T \,|\, \delta T> Q) \\
+ \xPr(\delta T\leq Q)\xPr_u(M\tau< T_v \leq M\tau + Q\,|\,\delta T\leq Q)
\end{multlined} \\
&\leq \varepsilon + \smallo(1) + \xPr_u(M\tau< T_v \leq M\tau + Q) \nonumber \;.
\end{align}
To estimate the latter, we write
\begin{align}
\label{eq:time-inhomogeneous:short-term-regularity:V:scaled}
\xPr_u(M\tau< T_v \leq M\tau + Q) &=
\xPr_u\big(\tau < T_v/M \leq \tau
+ Q/M\big) \;,
\end{align}
and note that $Q/M=\smallo(1)$.
Since $H$ is right-continuous, it follows that for every continuity point $\tau\in[0,\infty)$ of~$H$,
\begin{align}
\label{eq:time-inhomogeneous:short-term-regularity:V:limit}
\lim_{\substack{\lambda\isdef\lambda_{n(i)}\\ i\to\infty}}\xPr_u(M\tau< T_v \leq M\tau + Q) &=
H(\tau) - H(\tau+) = 0 \;.
\end{align}
Consequently, when $\tau\in[0,\infty)$ is a continuity point of~$H$,
\begin{align}
\limsup_{\substack{\lambda\isdef\lambda_{n(i)}\\ i\to\infty}}\xPr_u(S\leq M\tau<T_v) &\leq \varepsilon \;.
\end{align}
Since $\varepsilon>0$ is arbitrary, the limit exists and is zero, as it was claimed.
\end{proof}

\subsection{Proof of the main theorem}
\label{sec:proofmainth}

\begin{proof}[Proof of Theorem~\ref{thm:main}]
We consider the different scenarios and regimes separately.

\begin{enumerate}[label={(\roman*)}]
\item Scenario $M\prec\lambda$:
\begin{description}[topsep=0ex]
\item[\textsl{Supercritical regime:}] $M\check{\varepsilon}(0)\gamma(0)\succ 1$.\\
From~\eqref{eq:time-homogeneous:success-prob-vs-expectation}, we have $\xExp^{(0)}_u[T_v] \asymp \frac{1}{\check{\varepsilon}(0)\gamma(0)}\prec M$ as $\lambda\to\infty$.  Applying monotonicity and the Markov inequality, we obtain
\begin{align}
\xPr_u(T_v>M\tau) &\leq \xPr^{(0)}_u(T_v>M\tau) \leq\frac{1}{M\tau}\xExp^{(0)}_u[T_v] = \smallo(1)
\end{align}
for every $\tau>0$ as claimed.

\item[\textsl{Critical regime:}] $M\check{\varepsilon}(0)\gamma(0)\asymp 1$.\\
Since convergence in distribution comes from a metric, it is enough to show that for every sequence $(\lambda_n)_{n\in\NN}$ going to infinity, there exists a subsequence $(\lambda_{n(i)})_{i\in\NN}$ such that for every $\tau\in[0,\infty)$,
\begin{align}
\label{eq:main-result:proof}
\lim_{\substack{\lambda\isdef\lambda_{n(i)}\\ i\to\infty}} \xPr_u\Big(\frac{T_v}{M}>\tau\Big) 
\exp\left(\int_0^\tau M\check{\nu}(M\sigma)\,\dd\sigma\right) &= 1 \;.
\end{align}
Furthermore, in order to show the latter, it is enough to show that~\eqref{eq:main-result:proof} holds for a dense set of values $\tau\in[0,\infty)$.

We apply Proposition~\ref{prop:remaining-tasks}.
First note that, by Proposition~\ref{prop:critical-regime:orders-of-magnitude},
\begin{align}
\check{\varepsilon}(M\tau)\int_0^{M\tau}\gamma(s)\dd s &=
\check{\varepsilon}(M\tau)\int_0^\tau\gamma(M\sigma)M\dd\sigma \asymp M\check{\varepsilon}(0)\gamma(0)
\asymp 1.		
\end{align}
Therefore, it remains to verify that $C, k, m\in\ZZ^+$, $\delta s\in\RR^+$ and $\Xi_{M\tau}$ with $\xPr(\rv{\xi}\in\Xi_{M\tau})=\smallo(1)$ can be chosen such that Conditions~\ref{item:task:parameter:C}--\ref{item:task:parameter:E-vs-m} and \ref{item:task:truncated-mean}--\ref{item:task:last-excursion} are satisfied. We choose $C$, $k$, $m$ and $\delta s$ according to~\eqref{eq:parameters:explicit-choices:time-inhomogeneous}.  By the discussion after~\eqref{eq:parameters:explicit-choices:time-inhomogeneous}, these choices satisfy~\ref{item:task:parameter:C}--\ref{item:task:parameter:E-vs-m} at the critical regime. The short-term regularity Conditions~\ref{item:task:truncated-mean}--\ref{item:task:last-excursion} in turn are shown to be satisfied in Propositions~\ref{prop:time-inhomogeneous:short-term-regularity:I-II}--\ref{prop:time-inhomogeneous:short-term-regularity:V}.

\item[\textsl{Subcritical regime:}] $M\check{\varepsilon}(0)\gamma(0)\prec 1$.\\
According to Proposition~\ref{prop:critical-regime:orders-of-magnitude}, $\check{\varepsilon}(M\tau)\asymp\check{\varepsilon}(0)$ and $\gamma(M\tau)\asymp\gamma(0)$ as $\lambda\to\infty$.  Recall that under the measure $\xPr^{(M\tau)}_u(\cdot)$, the scaled hitting time $T_v/\xExp^{(M\tau)}_u[T_v]$ is asymptotically exponentially distributed as $\lambda\to\infty$ (see the review of the time-homogeneous results in Section~\ref{sec:intro:formulation}).  From~\eqref{eq:time-homogeneous:success-prob-vs-expectation}, we have $\xExp^{(M\tau)}_u[T_v] \asymp \frac{1}{\check{\varepsilon}(M\tau)\gamma(M\tau)}\succ M$ as $\lambda\to\infty$.  Applying monotonicity, for every $x>0$, we get
\begin{align}
\lim_{\lambda\to\infty}\xPr_u(T_v>M\tau) &\geq \lim_{\lambda\to\infty}\xPr^{(M\tau)}_u(T_v>M\tau) \\
&\geq \lim_{\lambda\to\infty}\xPr^{(M\tau)}_u\left(T_v>x\xExp^{(M\tau)}_u[T_v]\right)
= \ee^{-x} \;. \nonumber
\end{align}
Sending $x\to 0$, we find that $\lim_{\lambda\to\infty}\xPr_u(T_v>M\tau)=1$ as claimed.

\end{description}

\item Scenario $M\asymp\lambda$:
\begin{case}[1]{$0<\tau<\frac{c_U}{\mu_U}\frac{\lambda}{M}$.}\\
This case is similar to the scenario in which $M\prec\lambda$ (see above).
\end{case}
\begin{case}[2]{$\tau\geq\frac{c_U}{\mu_U}\frac{\lambda}{M}$.}\\
This case is similar to the scenario in which $M\succ\lambda$ (see below).
\end{case}
\item Scenario $M\succ\lambda$:\\
Note that in this scenario, for every $\sigma>0$, $\lambda_U(M\sigma)=0$ for all sufficiently large $\lambda$ while $\lambda_V(M\sigma)\to\infty$ as $\lambda\to\infty$.
Therefore, when $\lambda$ is sufficiently large, in order for the process to reach state $v$, it suffices that every particle on $U$ is removed and a particle is placed at each site in $V$.  To be specific, let us consider the process starting from $u$, and let $R_U$ denote the first time that every particle on $U$ is removed.  Note that $R_U$ is distributed as the maximum of $\abs{U}$ independent exponentially distributed random variables with rate~$1$, and in particular, the distribution of $R_U$ is independent of $\lambda$.  For each $b\in V$, let $S'_b$ and $R'_b$ denote respectively the first time after $R$ at which the birth clock or the death clock at site $b$ tick.  Note that $S'_b-R_U$ and $R'_b-R_U$ (for $b\in V$) are all independent and exponentially distributed, with $S'_b-R_U$ having rate~$\lambda_V$ and $R'_b-R_U$ having rate~$1$.  Let $\varepsilon>0$ be arbitrary.  Then, the probability that $\max_{b\in V}(S'_b-R_U)<\min\big(\min_{b\in V}(R'_b-R_U), \varepsilon\big)$ approaches~$1$ as $\lambda\to\infty$.  On the latter event, we clearly have $T_v<R_U+\varepsilon$.  It follow that
\begin{align}
\xPr_u(T_v>M\tau) &\leq \xPr_u(R_U+\varepsilon>M\tau) + \smallo(1) = \smallo(1),
\end{align}
proving the claim. 
\qedhere
\end{enumerate}
\end{proof}



\bibliographystyle{plainurl}

\bibliography{bibliography}


\end{document}